\documentclass{article}%
\usepackage{authblk}
\usepackage{amsmath}%
\usepackage{amssymb}%
\usepackage{stix}
\usepackage{graphicx}
\usepackage{hyperref}
\usepackage[usenames,dvipsnames]{color}
\usepackage{stackrel}
\allowdisplaybreaks

\newcommand{\abbr}[1]{{\sc\lowercase{#1}}}

\providecommand{\U}[1]{\protect\rule{.1in}{.1in}}
\newtheorem{theorem}{Theorem}

\newtheorem{condition}[theorem]{Condition}

\newtheorem{definition}[theorem]{Definition}

\newtheorem{lemma}[theorem]{Lemma}

\newtheorem{proposition}[theorem]{Proposition}
\newtheorem{remark}[theorem]{Remark}

\newenvironment{proof}[1][Proof]{\noindent\textbf{#1.} }{\ \rule{0.5em}{0.5em}}

\newcommand{\B}{\mathbb{B}}

\newcommand{\E}{\mathbb{E}}

\newcommand{\K}{\mathbb{K}}

\newcommand{\N}{\mathbb{N}}

	\renewcommand{\P}{\mathbb{P}}

\newcommand{\R}{\mathbb{R}}

\newcommand{\wt}{\widetilde}
\newcommand{\ovl}{\overline}
\newcommand{\ep}{\epsilon}
\newcommand{\cD}{\mathcal{D}}

\newcommand{\cF}{\mathcal{F}}

\newcommand{\cK}{\mathcal{K}}
\newcommand{\cL}{\mathcal{L}}
\newcommand{\cM}{\mathcal{M}}
\newcommand{\cN}{\mathcal{N}}

\newcommand{\cT}{\mathcal{T}}

\newcommand{\cY}{\mathcal{Y}}

\begin{document}
\title{Scaling limit for a second-order particle system\\
 with local annihilation
}
\author{Ruojun Huang\footnote{ Fachbereich Mathematik und Informatik, Universit\"at Münster. Einsteinstr. 62, 48149 M\"unster, Germany.  E-mail: ruojun.huang@uni-muenster.de.}}

\date{\today}
\maketitle

\abstract{For a second-order particle system in $\R^d$ subject to locally-in-space pairwise annihilation, we prove a scaling limit for its empirical measure on position and velocity towards a degenerate elliptic partial differential equation. Crucial ingredients are Green's function estimates for the associated hypoelliptic operator and an It\^o-Tanaka trick.}

\section{Introduction}
Second-order particle systems comprising of dynamics of positions and velocities, in the spirit of Newtonian mechanics, give a more complete description of interacting particles. When the behavior of a first-order system of positions only is understood, or is insufficient for one's purpose, it is natural to look at a second-order system. The technical difficulties associated with second-order systems are in general greater, some coming from the degeneracy of the associated operators. Here, we explore a simplest model of interacting particles {\it{with local interactions}} in $\R^d$ written in second-order, where the interaction occurs not in the equation of the motion but in pairwise annihilation. Besides being well-related to a large body of works on annihilating (or proliferating, coagulating, colliding etc) particle systems and their scaling limits (e.g. \cite{LX, ASZ, Oe, Oe2, MC, Rez, HR, FLO, FH}) and generally relevant for reaction-diffusion phenomena, one motivation is also from the technical side, namely, we aim to derive a scaling limit when the interaction range between particles is truly local (see Condition \ref{cond:local}).

Given a probability space $(\Omega,\cF,\P)$, for any positive integer $N\in\N$ and $d\ge 1$, consider a system of particles, whose cardinality at time $t=0$ is $N(0)=N$, and with positions and velocities at any time $t\ge0$ denoted by $x_i^N(t)\in\R^d$, $v_i^N(t)\in\R^d$, $i=1,...,N(t)$, and the interaction comes from the possibility of pairwise annhilation when spatial positions are close. More specifically, the position and velocity of each particle, while active, obey the second-order dynamics
\begin{align}\label{dynamics}
\begin{cases}
x_i^N(t) = x_i(0)+\int_0^t v_i^N(s)ds\\
v_i^N(t) = v_i(0) + B_i(t)
\end{cases}
,\quad t\ge 0, \, i\in\cN(t),
\end{align}
where $\cN(t)$ denotes the set of indices of active particles at time $t$ whose cardinality $|\cN(t)|=N(t)\le N$, $\{B_i(t)\}_{i=1}^\infty$ is a collection of independent standard Brownian motions in $\R^d$, and the initial conditions $(x_i(0),v_i(0))_{i=1}^\infty$ are random, chosen i.i.d. with some given probability density $f_0(x,v):\R^{2d}\to\R_+$ satisfying the condition below, independent of the Brownian motions.
\begin{condition}\label{cond:ini}
For every $d\ge1$, we assume that there exist some finite constants $\Gamma$ and $R$ such that $\|f_0(x,v)\|_{L^\infty(\R^{2d})}\le\Gamma$ and supp$(f_0)\subset \B_E((0,0), R)$, the closed Euclidean ball of radius $R$ centered at the origin in $\R^{2d}$. 
\end{condition}
The probability space is endowed with the canonical filtration 
\begin{align*}
\cF_t:=\sigma\big\{\{B_i(s): 0\le s\le t\}_{i=1}^\infty, (x_i(0),v_i(0))_{i=1}^\infty\big\}, \quad t\ge0.
\end{align*}
Interaction between different particles comes by means of pairwise annhilation, informally described as follows. Let $\theta(x):\R^d\to\R_+$ be given, nonnegative, of class $C^\infty$, with $\int\theta dx =1$, $\theta(0)=0$, $\theta(x)=\theta(-x)$ and compact support in $\B_E(0,1)$. Then, for any $\ep\in(0,1]$, denote by 
\begin{align*}
\theta^\ep(x):=\ep^{-d}\theta(\ep^{-1}x).
\end{align*}
(For the necessity of the condition $\theta(0)=0$, see Remark \ref{rmk:theta}.)
A configuration of the particle system is given by 
\begin{align*}
\zeta = (x_1,v_1, x_2, v_2, ...,x_N, v_N) \in(\R^d\cup\emptyset)^{2N}
\end{align*}
where we adopt the convention that if a particle $i_0$ has been annihilated, we set $x_{i_0}=v_{i_0}=\emptyset$. It will no more participate in the dynamics of the system. Now suppose the current configuration is $\zeta$, then for any pair of particles $i,j$ with $i \neq j$ (each ranging over the index set $\cN(\zeta)$ of active particles), with an infinitesimal rate 
\begin{align*}
\frac{1}{N}\theta^\ep(x_i-x_j)
\end{align*}
we remove both particles $i,j$ from the system (i.e. they annihilate with each other) and obtain a new configuration, denoted $\zeta^{-ij}$ which is equal to $\zeta$ in all other slots except for $x_i,v_i,x_j,v_j$ which have been set to $\emptyset$. Since the support of the function $\theta^\ep$ has radius $\ep$, interaction happens for the pair $i,j$ only when their spatial distance is within the critical distance $\ep$. In fact, our interest is in the case where $\ep=\ep(N)$ depends on $N$ and converges to zero as $N\to\infty$.  An important case to cover is the regime of local interaction, corresponding to $\ep = N^{-1/d}$ in which case an individual particle typically interacts with a bounded number of others at any given time. Contrary to lattice models \cite{DP, KL} where nearest-neighbor or bounded-range interactions are standard, in our continuum case such local interaction has to be enforced by a potential (i.e. $\theta^\ep$), which is a source of technical difficulty. One of the key techniques available to date to handle such local interaction, or sometimes even more stringently the mean-free path interaction, is due to Hammond and Rezakhanlou \cite{HR}, see also \cite{LX, ASZ, Va, Uc} for other techniques. On the other hand, mean-field type models correspond to choosing $\ep$ independent of $N$, and in that setting techniques for proving scaling limits of (both first and second-order) particle systems have become very advanced, cf. the survey \cite{JW}, \cite{BJS}, even when interactions are very singular and influence directly the stochastic differential equations. Clearly, there is a big technical gap between mean-field and more local cases, which serves as another motivation for this research.

\begin{condition}\label{cond:local}
For any $d\ge1$, we assume that $\ep=\ep(N)\to0$ as $N\to\infty$ and 
\begin{align*}
\limsup_{N\to\infty}\frac{\ep^{-d}}{N}\le 1.
\end{align*}
When the equality is achieved, we call the interaction local.
\end{condition}
We now give formally the infinitesimal generator of the jointly Markov process $(x_i^N(t), v_i^N(t))_{i=1}^{\cN{(t)}}$:
\begin{align}\label{gen}
\cL_N=\cL_N^D+\cL_N^J
\end{align}
where the diffusion and respectively, jump (i.e. annhilation) part of the generator are (acting on test functions $F$)
\begin{align*}
\cL_N^DF(\zeta)&:=\sum_{i\in\cN(\zeta)}\big[v_i\cdot \nabla_{x_i}F(\zeta)+\frac{1}{2}\Delta_{v_i}F(\zeta)\big] \\
\cL_N^JF(\zeta)&:=\sum_{i\in\cN(\zeta)}\sum_{j\in\cN(\zeta): j\neq i}\frac{1}{N}\theta^\ep(x_i-x_j)\big[F(\zeta^{-ij})-F(\zeta)\big].  
\end{align*} 
For every $N$ and $T$, denote the process of joint empirical measure of positions and velocities of active particles
\begin{align}\label{empirical}
\mu_t^N(dx,dv):=\frac{1}{N}\sum_{i\in\cN(t)}\delta_{\left(x_i^N(t),\, v_i^N(t)\right)}(dx, dv), \quad t\in[0,T].
\end{align}
For each $t$, $\mu_t^N(dx,dv)$ is a subprobability measure on $\R^{2d}$. Denote by $\cM_{+,1}(\R^{2d})$ the space of subprobability measures on $\R^{2d}$ endowed with the weak topology, and by 
\begin{align}\label{skorohod}
\cD_T(\cM_{+,1}):=\cD\big([0,T];\cM_{+,1}(\R^{2d})\big)
\end{align}
the space of c\`adl\`ag functions on $[0,T]$ taking values in $\cM_{+,1}(\R^{2d})$ endowed with the Skorohod topology \cite[Chapter 3]{EK}. Our goal is to prove that the $N\to\infty$ limit of \eqref{empirical} is given by the unique weak solution (see Definition \ref{def:weak}) of the following degenerate elliptic partial differential equation (\abbr{PDE})
\begin{align}\label{pde}
\begin{cases}
\partial_t f(t,x,v)& =- v\cdot\nabla_x f + \frac{1}{2}\Delta_v f - 2\int_{\R^d} f(t,x,v) f(t,x,v') dv'\\
f|_{t=0}& = f_0(x,v)
\end{cases}
, \quad (t,x,v)\in[0,T]\times\R^d\times\R^d.
\end{align}

\begin{theorem}\label{thm:main}
Under Conditions \ref{cond:ini} and \ref{cond:local}, for any finite $T$ and $d\ge1$, the sequence of empirical measures $\{\mu^N_t:t\in[0,T]\}_{N\in\N}$ converges in probability in the space $\cD_T(\cM_{+,1})$ towards a limit $\{\ovl{\mu}_t:t\in[0,T]\}$. For every $t\in[0,T]$, the measure $\ovl{\mu}_t(dx,dv)$ has a density $f(t,x,v)$ with respect to Lebesgue measure and $\{f(t,x,v):t\in[0,T],x,v\in\R^d\}$ is the unique weak solution to the \abbr{PDE} \eqref{pde} in the sense of Definition \ref{def:weak}.
\end{theorem}

\begin{remark}
Our setting \eqref{dynamics} of a second-order system is almost the simplest possible. The motivation and challenge is to handle the local interaction. One may notice that, if substituting the velocity into the first equation \eqref{dynamics}, then we get a closed expression for the dynamics of the position
\begin{align}\label{integrated}
x_i^N(t) = x_i(0)+v_i(0)t + \int_0^t B_i(s)ds. 
\end{align}
One may ask why not study directly a first-order system of $N$ locally interacting particles with dynamics given by \eqref{integrated}, especially if one is mainly interested in the limiting density of positions alone
\begin{align*}
\rho(t,x): = \int_{\R^d}f(t,x,v)dv.
\end{align*}
However, it is not hard to check that $t\mapsto\int_0^tB(s)ds$ is not a Markov process, hence working directly with the first-order dynamics is not necessarily easier. Consistent with that, if we integrate the \abbr{PDE} \eqref{pde} formally in variable $v$, we do not obtain a closed equation for $\rho(t,x)$. It remains unclear what (if any) equation $\rho$ should satisfy. 
\end{remark}

%\begin{remark}
%After the completion of this paper, we became aware of the article of Rezakhanlou \cite{Rez} proving a system of hard spheres with stochastic collision rule, but deterministic motion in between collisions, converges to the Boltzmann equation, a famous and far more challenging kinetic equation. The setting is already quite different, and we point out several differences. In \cite{Rez} as in almost all investigations of hard sphere models, the velocity of a particle between collisions is constant, thus their system is without the regularizing effect of Brownian motion (or equivalently, the Laplacian), and the It\^o-Tanaka trick is not applicable. Secondly, the pairwise collision rate employed in \cite[page 554]{Rez} is more relaxed than our scaling \eqref{scaling}. Indeed, one can rewrite their rate in exactly the same form $N^{-1}\theta^\ep(x_i-x_j)$ \eqref{rate} (multiplied by another order $1$ cross section term; they use the letter $V$ for the fixed bump function $\theta$), and check that they require $\ep\ge N^{-1/(2d)}$.
%\end{remark}

The rest of the article is organized as follows. In Section \ref{sec:emp}, we derive an identity satisfied by the empirical measure, and identity the nonlinear term whose convergence is the main issue. For this purpose, in Section \ref{sec:ito} we set up the It\^o-Tanaka trick which involves the solution of an auxiliary \abbr{PDE} with a degenerate Kolmogorov (or Fokker-Planck) operator, see that section for more on the meaning of this manouvre. It turns out that key to our proof are pointwise estimates of the Green's function associated with this operator, which are recalled and developed in Section \ref{sec:hypoe}. In Section \ref{sec:control}, we apply these estimates to show that various terms coming from the It\^o-Tanaka trick are negligible. In Section \ref{sec:tight}, we check the tightness of the laws of our particle system, which together with previous sections completes the proof of subsequential convergence. In the last Section \ref{sec:unique}, we prove the uniqueness of the limit \abbr{PDE}. The techniques used here, which are very influenced by the paper of Hammond-Rezakhanlou \cite{HR} on particle approximation of the Smoluchowski coagulation equation, share elements with the author's previous works \cite{FH, FH2}, hence certain arguments are made more succinct. 

\section{Identity from the empirical measure}\label{sec:emp}
Fixing $T<\infty$ and any $\phi_t(x,v)=\phi(t,x,v)\in C_c^\infty ([0,T]\times\R^{2d})$, we consider the process
\begin{align}\label{sum-emp}
F_1(t, \zeta(t)):=\langle \phi_t, \mu_t^N\rangle = \frac{1}{N}\sum_{i\in\cN(t)}\phi\big(t, x_i^N(t),v_i^N(t)\big), \quad t\ge0,
\end{align}
where $\langle f, \mu\rangle := \int_{\R^{2d}} f(x,v) \mu(dx,dv)$. Applying It\^o's formula to \eqref{sum-emp}, we obtain
\begin{equation}\label{eq-for-measure}
\begin{aligned}
\langle \phi_T, \mu_T^N\rangle&=F_1(T,\zeta(T)  =F_1(0,\zeta(0)+\int_0^T\left(\partial_t + \cL_N^D+\cL_N^J\right)F_1(t,\zeta(t))dt+
M_T^{N,\phi,D}+M_T^{N,\phi,J}\\
&= \langle \phi_0, \mu_0^N\rangle + \int_0^T \left\langle \partial_t\phi+v\cdot \nabla_x\phi +\frac{1}{2} \Delta_v\phi,\, \mu_t^N\right\rangle dt\\
& -\int_0^T\frac{1}{N^2}\sum_{i\in\cN(t)}\sum_{j\in\cN(t):j\neq i}\theta^\ep\big(x_i^N(t)-x_j^N(t)\big)\left[\phi\big(t, x_i^N(t),v_i^N(t)\big)+\phi\big(t, x_j^N(t),v_j^N(t)\big)\right] \, dt \\
&+ M_T^{N,\phi,D}+M_T^{N,\phi,J}.
\end{aligned}
\end{equation}
More precisely, \eqref{sum-emp} is a time-dependent functional of the particle configuration $\zeta(t)$, hence we may act the Markov generator \eqref{gen} on \eqref{sum-emp}. The diffusion part of the generator acts individually on each $\phi\left(t, x_i^N(t), v_i^N(t)\right)$  as
\begin{align*}
\cL_N^D\phi\left(t, x_i^N(t), v_i^N(t)\right)=\left(v_i^N(t)\cdot \nabla_x +\frac{1}{2} \Delta_v\right)\phi\big(t, x_i^N(t),v_i^N(t)\big),
\end{align*}
hence
\begin{align*}
\left(\partial_t+\cL_N^D\right)F_1(t,\zeta(t))&=\frac{1}{N}\sum_{i\in\cN(t)}\left(\partial_t+v_i^N(t)\cdot \nabla_x +\frac{1}{2} \Delta_v\right)\phi\big(t, x_i^N(t),v_i^N(t)\big)\\
&=\left\langle \partial_t\phi+v\cdot \nabla_x\phi +v\cdot \nabla_x\phi +\frac{1}{2} \Delta_v\phi,\, \mu_t^N\right\rangle .
\end{align*}
Turning to the jump part of the generator, for each distinct pair $i\neq j \in\cN(t)$, we know that at an exponential rate 
\begin{align}\label{pair-rate}
q_{ij}^N(t):=\frac{1}{N}\theta^\ep\big(x_i^N(t)-x_j^N(t)\big),
\end{align}
they annihilate and hence the sum \eqref{sum-emp} decreases by the amount
\[
F_1(t, \zeta^{-ij}(t))-F_1(t, \zeta(t))= -N^{-1}\big[\phi\big(t, x_i^N(t),v_i^N(t)\big)+\phi\big(t, x_j^N(t),v_j^N(t)\big)\big]. 
\]
Therefore, 
\[
\cL_N^JF_1(t, \zeta(t))=-\sum_{i\in\cN(t)}\sum_{j:\, j\neq i}\frac{1}{N}\theta^\ep\big(x_i^N(t)-x_j^N(t)\big)\frac{1}{N}\big[\phi\big(t, x_i^N(t),v_i^N(t)\big)+\phi\big(t, x_j^N(t),v_j^N(t)\big)\big].
\]
(From now on, for the ease of notation, all sums are only over active particles hence we sometimes drop $\cN(t)$.)
Any inactive particle is assigned a cemetery state $\emptyset$; it stays put and does not contribute to the dynamics. We note in passing that even though $v\in\R^d$ is an unbounded variable, $v\cdot \nabla_x\phi$ is bounded continuous since here $v$ is in the compact support of $\phi$, thus in the first line of \eqref{eq-for-measure}, we have the form of $\mu_t^N$ integrated against a bounded continuous test function.

To compensate for the action of the generator, $M_t^{N,\phi,D}, M_t^{N,\phi,J}$ are two martingales due to diffusion and jump (i.e. annhilation), respectively. We can bound the martingales via quadratic variation: indeed,
\begin{align}\label{mg-vanish-1}
\E\big[|M_T^{N,\phi,D}|^2\big]&=\E\Big[\big|\frac{1}{N}\int_0^T\sum_{i\in\cN(t)}\nabla_{v}\phi\big(t, x_i^N(t),v_i^N(t)\big)\cdot dB_i(t)\big|^2\Big]\nonumber\\
&=\frac{1}{N^2}\E\int_0^T\sum_{i\in\cN(t)}\big|\nabla_{v}\phi\big(t, x_i^N(t),v_i^N(t)\big)\big|^2dt\nonumber\\
&\le \frac{\|\phi\|_{C^1}^2 }{N^2}\E\int_0^T\cN(t)dt\le \frac{\|\phi\|_{C^1}^2T}{N},
\end{align}
and by \cite[Proposition 8.7]{DN} (the quantity $\alpha$ therein, so-called carr\'e-du-champ, is the sum over all possible jumps of the particle configuration $\zeta(t)$, of the squared change of the value of the functional $F_1(t, \zeta(t))$, weighted by the rate $q_{ij}^N(t)$ of such a change), see also \cite[Appendix 1, Lemma 5.1]{KL},
\begin{align*}
\E\big[|M_T^{N,\phi,J}|^2\big]&\le \E\int_0^T\sum_{i\in\cN(t)}\sum_{j:j\neq i} q_{ij}^N(t)\left[F_1(t, \zeta^{-ij}(t))-F_1(t, \zeta(t))\right]^2dt\\
&\le\E\int_0^T \sum_{i\in\cN(t)}\sum_{j:j\neq i}\frac{1}{N}\theta^\ep\big(x_i^N(t)-x_j^N(t)\big)\left[\frac{1}{N}\phi\big(t, x_i^N(t),v_i^N(t)\big)+\frac{1}{N}\phi\big(t, x_j^N(t),v_j^N(t)\big)\right]^2dt\\
&=\frac{\|\phi\|_{L^\infty}^2}{N^3}\E\int_0^T\sum_{i\in\cN(t)}\sum_{j:j\neq i}\theta^\ep\big(x_i^N(t)-x_j^N(t)\big)dt.
\end{align*}
To bound the remaining quantity, applying It\^o's formula to $\cN(t)$ and then taking expectation, we obtain (cf. \cite[Proof of Lemma 3.1]{HR} for a similar context)
\begin{align*}
\E\cN(T)=\E\cN(0)-\E\int_0^T\sum_{i\in\cN(t)}\sum_{j:j\neq i}\frac{2}{N}\theta^\ep\big(x_i^N(t)-x_j^N(t)\big)dt
\end{align*}
hence we see that 
\begin{align}\label{total-card}
\E\int_0^T\sum_{i\in\cN(t)}\sum_{j:j\neq i}\theta^\ep\big(x_i^N(t)-x_j^N(t)\big)dt\le \frac{N^2}{2},
\end{align}
whereby
\begin{align}\label{mg-vanish-2}
\E\big[|M_T^{\phi,J}|^2\big]\le \frac{\|\phi\|_{L^\infty}^2}{2N}.
\end{align}
To handle the nonlinearity, we first fix another test function $\psi(x,v)\in C_c^\infty(\R^{2d})$ and consider 
\begin{align}\label{nonlinear}
\int_0^T\frac{1}{N^2}\sum_{i\in\cN(t)}\sum_{j:j\neq i}\theta^\ep\big(x_i^N(t)-x_j^N(t)\big)\phi\big(t, x_i^N(t),v_i^N(t)\big) \psi\big(x_j^N(t),v_j^N(t)\big)\, dt.
\end{align}
Let us denote by $K=K^{\phi,\psi}$ the maximal radius of the compact supports of $\phi$ and $\psi$, and by
\begin{align}\label{compact}
\chi^K(\xi):\R^d\to[0,1] 
\end{align}
a compactly-supported function of class $C^\infty$ that is identically $1$ for $|\xi|\le 2K$ and $0$ for $|\xi|\ge 4K$. Then, \eqref{nonlinear} is equal to 
\begin{align}\label{truncated}
\int_0^T\frac{1}{N^2}\sum_{i\in\cN(t)}\sum_{j:j\neq i}\theta^\ep\big(x_i^N(t)-x_j^N(t)\big)\chi^K\big(v_i^N(t)-v_j^N(t)\big)\phi\big(t, x_i^N(t),v_i^N(t)\big) \psi\big(x_j^N(t),v_j^N(t)\big)\, dt
\end{align}
Indeed, since $v_i^N(t), v_j^N(t)$ are in the arguments of $\phi,\psi$, their moduli are less than $K$, whereby $|v_i^N(t)- v_j^N(t)|\le 2K$ otherwise one of $\phi,\psi$ is zero. Hence, adding the function $\chi^K$ to \eqref{nonlinear} does not change its value. 

\medskip
In the next two sections we concentrate on proving
\begin{proposition}\label{ppn:shift}
For any finite $T$ and $d\ge1$,
\begin{align*}
\lim_{|z|\to0}\limsup_{N\to\infty}\E\Big|\int_0^T\frac{1}{N^2}\sum_{i\in\cN(t)}\sum_{j: j\neq i}&\big[\theta^\ep\big(x_i^N(t)-x_j^N(t)+z\big)-\theta^\ep\big(x_i^N(t)-x_j^N(t)\big)\big]\\
&\cdot\chi^K\big(v_i^N(t)-v_j^N(t)\big)\phi\big(t, x_i^N(t),v_i^N(t)\big) \psi\big(x_j^N(t),v_j^N(t)\big)\, dt\Big|=0.
\end{align*}
\end{proposition}

This continuity claim should be seen as a form of local equilibrium (cf. \cite{KL}) for our continuum particle system at two macroscopically separated sites (here $z$ is a macroscopic variable whereas $\ep$ is microscopic or mesoscopic). Since for $\ep$ small, the function $\theta^\ep$ is very close to a Dirac delta function (hence singular), this continuity claim is difficult to prove directly. The main bulk of subsequent work is to apply the so-called It\^o-Tanaka trick to \eqref{truncated}, so that instead of working with the function $\theta^\ep$, we work equivalently with a more regular function $U^\ep$ (more regular in the sense not for fixed $\ep$ but for $\ep$ tending to zero). Roughly speaking, it is given by a convolution of $\theta^\ep$  with the Green's function of a hypoelliptic operator, see \eqref{gr-conv}. Since the Green's function is not integrable at infinity and $\theta^\ep$ is a function of $x$ only, we seek additional localization in the variable $v$. This is the reason we introduce the function $\chi^K$, as it will be part of an auxiliary \abbr{PDE} \eqref{ito-tanaka} we set up for $U^\ep$.

\begin{remark}\label{rmk:theta}
By the assumption $\theta(0)=0$, we can have a version of Proposition \ref{ppn:shift} where the double sum is over all $i,j$ (including repeated indices $i=j$). Indeed, there are at most $N$ diagonal terms, and after bounding the test functions by constants, we would like their contributions, which is at most  
\[
\lim_{|z|\to0}\limsup_{N\to\infty}\frac{C}{N^2}N\ep^{-d}|\theta(\ep^{-1}z)-\theta(0)|,
\]
to be zero. Recall that at the local scaling, $\ep^{-d}=N$ and if $\theta(0)>0$ the above quantity is not negligible. If $\theta(0)=0$, then since $\ep=\ep(N)\to0$ first and $\theta$ has compact support, we have $\theta(\ep^{-1}z)\to0$ as $\ep\to0$ for any fixed $z$. The above quantity is thus $0$.
\end{remark}

\medskip
Given Proposition \ref{ppn:shift}, one can argue as in \cite[pp. 42-43]{HR} or \cite[pp. 631-633]{FH} to derive the following Proposition \ref{ppn:conv-nonlin}, that is key to showing the convergence of the nonlinearity. Let $\eta(x):\R^d\to\R_+$ be given, nonnegative, of class $C^\infty$ with $\int\eta dx=1$, $\eta(x)=\eta(-x)$ and compactly supported in $\B_E(0,1)$. Then, for any $\delta\in(0,1]$, denote by 
\begin{align*}
\eta^\delta(x):=\delta^{-d}\eta(\delta^{-1}x).
\end{align*}
\begin{proposition}\label{ppn:conv-nonlin}
For any finite $T$ and $d\ge1$,
\begin{align*}
\lim_{\delta\to0}\limsup_{N\to\infty}&\, \E\Big|\int_0^T\frac{1}{N^2}\sum_{i\in\cN(t)}\sum_{j\in\cN(t):j\neq i}\theta^\ep\big(x_i^N(t)-x_j^N(t)\big)
\phi\big(t, x_i^N(t),v_i^N(t)\big) \psi\big(x_j^N(t),v_j^N(t)\big)\, dt\\
&-\int_0^T\int_{\R^{d}}dw\left\langle\eta^\delta(w-x_1)\eta^\delta(w-x_2)\phi(t, w,v_1)\psi(w,v_2), \;\mu_t^N(dx_1,dv_1)\mu_t^N(dx_2,dv_2)\right\rangle dt
\Big|=0.
\end{align*}
\end{proposition}

\begin{proof}
We shorten some parts of the proof as it is similar to arguments detailed in \cite{HR, FH, FH2}. 
In the sequel, we use Err$(N,\delta)$ to denote a stochastic error that can change from line to line and which vanishes in the following limit
\begin{align*}
\lim_{\delta\to0}\limsup_{N\to\infty}\E|\text{Err}(N,\delta)|=0.
\end{align*}
By the assumption $\theta(0)=0$, Proposition \ref{ppn:shift} (we can remove the function $\chi^K$ without changing the value, as explained near \eqref{truncated}) and Remark \ref{rmk:theta}, we firstly have that 
\begin{align*}
&\int_0^T\frac{1}{N^2}\sum_{i\in\cN(t)}\sum_{j:j\neq i}\theta^\ep\big(x_i^N(t)-x_j^N(t)\big)\phi\big(x_i^N(t),v_i^N(t)\big) \psi\big(x_j^N(t),v_j^N(t)\big)\, dt\\
&=\int_0^T\frac{1}{N^2}\sum_{i,j\in\cN(t)}\theta^\ep\big(x_i^N(t)-x_j^N(t)\big)\phi\big(x_i^N(t),v_i^N(t)\big) \psi\big(x_j^N(t),v_j^N(t)\big)\, dt\\
&=\int_0^T\iint_{\R^{2d}}\eta^\delta(z_1)\eta^\delta(z_2)dz_1dz_2\frac{1}{N^2}\sum_{i,j\in\cN(t)}\theta^\ep\big(x_i^N(t)-x_j^N(t)-z_1+z_2\big)
\\
&\quad\quad\quad   \cdot\phi\big(t, x_i^N(t),v_i^N(t)\big) \psi\big(x_j^N(t),v_j^N(t)\big)\, dt +\text{Err}(N,\delta).
\end{align*}
Indeed, any $z_1,z_2$ in the support of $\eta^\delta$ in the last line have modulus less than $\delta$, hence the stochastic error depends only on $N$ and $\delta$. Next, since $\phi,\psi$ are test functions, shifting their $x$-slot by $z_1,z_2$ respectively causes a stochastic error on the order of $\delta$ (here recall \eqref{total-card}), hence the preceding display is equal to
\begin{align*}
&\int_0^T\iint_{\R^{2d}}\eta^\delta(z_1)\eta^\delta(z_2)dz_1dz_2\frac{1}{N^2}\sum_{i,j\in\cN(t)}\theta^\ep\big(x_i^N(t)-x_j^N(t)-z_1+z_2\big)
\\
&\quad\quad\quad   \cdot\phi\big(t, x_i^N(t)-z_1,v_i^N(t)\big) \psi\big(x_j^N(t)-z_2,v_j^N(t)\big)\, dt +\text{Err}(N,\delta).
\end{align*}
By a change of variables $w_1=x_i^N(t)-z_1$, $w_2=x_j^N(t)-z_2$, we can rewrite the above as
\begin{align*}
&\int_0^T\iint_{\R^{2d}}\frac{1}{N^2}\sum_{i,j\in\cN(t)}\eta^\delta\left(x_i^N(t)-w_1\right)\eta^\delta\big(x_j^N(t)-w_2\big)\theta^\ep\big(w_1-w_2\big)
\\
&\quad\quad\quad  \cdot\phi\big(t, w_1,v_i^N(t)\big) \psi\big(w_2,v_j^N(t)\big)\, dw_1dw_2dt +\text{Err}(N,\delta).
\end{align*}
Since $|w_1-w_2|<\ep$ in the support of $\theta^\ep$, we shift the $w_2$ in $\psi\big(w_2,v_j^N(t)\big)$ to $w_1$ causing an error on the order of $\ep\delta^{-2d}$. Since in our claim $\ep\to0$ before $\delta\to0$, this error is negligible and can be combined into Err$(N,\delta)$. Similar, shifting the $w_2$ in the argument of $\eta^\delta\big(x_j^N(t)-w_2\big)$ to $w_1$ causes an error of order $\ep\delta^{-1-2d}$, also negligible. Thus the preceding display is equal to
\begin{align*}
&\int_0^T\iint_{\R^{2d}}\frac{1}{N^2}\sum_{i,j\in\cN(t)}\eta^\delta\left(x_i^N(t)-w_1\right)\eta^\delta\big(x_j^N(t)-w_1\big)\theta^\ep\big(w_1-w_2\big)
\\
&\quad\quad\quad \cdot\phi\big(t, w_1,v_i^N(t)\big) \psi\big(w_1,v_j^N(t)\big)\, )dw_1dw_2dt +\text{Err}(N,\delta)\\
&=\int_0^T\int_{\R^{d}}\frac{1}{N^2}\sum_{i,j\in\cN(t)}\eta^\delta\left(x_i^N(t)-w_1\right)\eta^\delta\big(x_j^N(t)-w_1\big)\phi\big(t, w_1,v_i^N(t)\big) \psi\big(w_1,v_j^N(t)\big)\, dw_1dt +\text{Err}(N,\delta)\\
&=\int_0^T\int_{\R^{d}}dw_1\left\langle\eta^\delta(x_1-w_1)\eta^\delta(x_2-w_1)\phi(t, w_1,v_1)\psi(w_1,v_2), \;\mu_t^N(dx_1,dv_1)\mu_t^N(dx_2,dv_2)\right\rangle dt
+\text{Err}(N,\delta),
\end{align*}
where one of the integral in $dw_2$ is removed using the property $\int\theta^\ep(w_1-w_2)dw_2=1$ for any $w_1$.
%where the second equality is due to a shift of the arguments of $\phi,\psi$ by $z_1,z_2$ causing an additional error on the order of $\delta$, the third equality is a change of variables, the fourth equality is a shift in the arguments of $\eta^\delta$ and $\psi$ from $w_2$ to $w_1$ causing an additional error on the order of $\ep\delta^{-2d-1}$, the fifth equality is due to $\int\theta^\ep(w_1-w_2)dw_2=1$ for any $w_1$, and 
The last equality is a rewritting in terms of the empirical measure $\mu^N_t$. 
\end{proof}

\section{The It\^o-Tanaka trick}\label{sec:ito}
In this section, we apply the so-called It\^o-Tanaka trick to \eqref{truncated}. The trick itself is a classical one in stochastic analysis, related to Tanaka's formula in local time, see e.g. \cite[pp. 4]{FGP} for an account, nevertheless its use in a particle system context may be quite recent, which we learned from Hammond and Rezakhanlou \cite{HR}. Roughly speaking, its aim is to replace a term of low regularity equivalently by a series of other terms involving functions of better regularity, by applying inversely the It\^o's formula.

For the above-fixed test functions $\phi,\psi$, and for any fixed $z\in\R^d$, consider the process
\begin{align}\label{aux-double-sum}
\cY_t:=\frac{1}{N^2}\sum_{i\in\cN(t)}\sum_{j\in\cN(t):\,j\neq i}H^{\ep,z}\big(x_i^N(t)-x_j^N(t), v_i^N(t)-v_j^N(t)\big)\phi\big(t, x_i^N(t),v_i^N(t)\big) \psi\big(x_j^N(t),v_j^N(t)\big), \quad t\ge0
\end{align}
which is also a functional of the configuration $\zeta(t)$, 
where 
\begin{align}\label{def:H}
H^{\ep,z}(x,v):=U^\ep(x+z,v)-U^\ep(x,v)
\end{align}
and $U^{\ep}(x,v):\R^{2d}\to\R$ is a function of class $C^\infty$ to be defined in \eqref{gr-conv}. Applying It\^o's formula to $\cY_t$, we obtain that 

\begin{align}
&\cY_T-\cY_0 \label{ini-term}\\
&=\int_0^T\frac{1}{N^2}\sum_{i\in\cN(t)}\sum_{j:j\neq i}\big(v_i^N(t)-v_j^N(t)\big)\cdot\nabla_xH^{\ep,z}\big(x_i^N(t)-x_j^N(t), v_i^N(t)-v_j^N(t)\big)\phi\big(t, x_i^N(t),v_i^N(t)\big) \psi\big(x_j^N(t),v_j^N(t)\big)\, dt \label{2nd-order-1}\\
&+\int_0^T\frac{1}{N^2}\sum_{i\in\cN(t)}\sum_{j:j\neq i}\Delta_vH^{\ep,z}\big(x_i^N(t)-x_j^N(t), v_i^N(t)-v_j^N(t)\big)\phi\big(t, x_i^N(t),v_i^N(t)\big) \psi\big(x_j^N(t),v_j^N(t)\big)\, dt \label{2nd-order-2}\\
&+\int_0^t\frac{1}{N^2}\sum_{i\in\cN(s)}\sum_{j:j\neq i}\partial_t\phi\big(t, x_i^N(s),v_i^N(s)\big) \psi\big(x_j^N(s),v_j^N(s)\big)H^{\ep,z}\big(x_i^N(t)-x_j^N(t), v_i^N(t)-v_j^N(t)\big)\, ds \label{double-sum-0}\\
&+\int_0^T\frac{1}{N^2}\sum_{i\in\cN(t)}\sum_{j:j\neq i}\Big[v_i^N(t)\cdot\nabla_x\phi\big(t, x_i^N(t),v_i^N(t)\big) \psi+v_j^N(t)\cdot\nabla_x\psi\big(x_j^N(t),v_j^N(t)\big)\phi\Big]H^{\ep,z}\, dt \label{double-sum-1}\\
&+\int_0^T\frac{1}{2N^2}\sum_{i\in\cN(t)}\sum_{j:j\neq i}\Big[\Delta_v\phi\big(t, x_i^N(t),v_i^N(t)\big) \psi\big(x_j^N(t),v_j^N(t)\big)+\Delta_v\psi\big(x_j^N(t),v_j^N(t)\big)\phi\big(t, x_i^N(t),v_i^N(t)\big)\Big]H^{\ep,z}\, dt \label{double-sum-2}\\
&-\int_0^T\frac{2}{N^3}\sum_{i\in\cN(t)}\sum_{k:k\neq i}\theta^\ep\big(x_i^N(t)-x_k^N(t)\big)\sum_{j:j\neq i,k}H^{\ep,z}\big(x_i^N(t)-x_j^N(t), v_i^N(t)-v_j^N(t)\big)\phi\big(t, x_i^N(t),v_i^N(t)\big) \psi\big(x_j^N(t),v_j^N(t)\big)\, dt \label{triple-sum-1}\\
&-\int_0^T\frac{2}{N^3}\sum_{i\in\cN(t)}\sum_{k:k\neq i}\theta^\ep\big(x_i^N(t)-x_k^N(t)\big)\sum_{j:j\neq i,k}H^{\ep,z}\big(x_j^N(t)-x_i^N(t), v_j^N(t)-v_i^N(t)\big)\phi\big(t, x_j^N(t),v_j^N(t)\big) \psi\big(x_i^N(t),v_i^N(t)\big)\, dt \label{triple-sum-2}\\
&-\int_0^T\frac{2}{N^3}\sum_{i\in\cN(t)}\sum_{k:k\neq i}\theta^\ep\big(x_i^N(t)-x_k^N(t)\big)H^{\ep,z}\big(x_i^N(t)-x_k^N(t), v_i^N(t)-v_k^N(t)\big)\phi\big(t, x_i^N(t),v_i^N(t)\big) \psi\big(x_k^N(t),v_k^N(t)\big)\, dt \label{triple-sum-3}\\
&+\wt M_T^{D}+\wt M_T^J, \label{mg-terms}
\end{align}
where we omitted some arguments for brevity as they are self-evident. Specifically, in \eqref{triple-sum-1}-\eqref{triple-sum-3} we used the fact that if a pair $(i,k)$ of distinct particles coagulate at time $t$ (which occurs at rate $q_{ik}^N(t)$ \eqref{pair-rate}), the double sum \eqref{aux-double-sum} decreases by an amount of 
\begin{align*}
&\frac{1}{N^2}\sum_{j\neq i, k}H^{\ep, z}\left(x_i^N(t)-x_j^N(t), v_i^N(t)-v_j^N(t)\right)\phi\big(t, x_i^N(t),v_i^N(t)\big) \psi\big(x_j^N(t),v_j^N(t)\big)\\
&+\frac{1}{N^2}\sum_{j\neq i, k}H^{\ep, z}\left(x_k^N(t)-x_j^N(t), v_k^N(t)-v_j^N(t)\right)\phi\big(t, x_k^N(t),v_k^N(t)\big) \psi\big(x_j^N(t),v_j^N(t)\big)\\
&+\frac{1}{N^2}\sum_{j\neq i, k}H^{\ep, z}\left(x_j^N(t)-x_i^N(t), v_j^N(t)-v_i^N(t)\right)\phi\big(t, x_j^N(t),v_j^N(t)\big) \psi\big(x_i^N(t),v_i^N(t)\big)\\
&+\frac{1}{N^2}\sum_{j\neq i, k}H^{\ep, z}\left(x_j^N(t)-x_k^N(t), v_j^N(t)-v_k^N(t)\right)\phi\big(t, x_j^N(t),v_j^N(t)\big) \psi\big(x_k^N(t),v_k^N(t)\big)\\
&+\frac{1}{N^2}H^{\ep, z}\left(x_i^N(t)-x_k^N(t), v_i^N(t)-v_k^N(t)\right)\phi\big(t, x_i^N(t),v_i^N(t)\big) \psi\big(x_k^N(t),v_k^N(t)\big)\\
&+\frac{1}{N^2}H^{\ep, z}\left(x_k^N(t)-x_i^N(t), v_k^N(t)-v_i^N(t)\right)\phi\big(t, x_k^N(t),v_k^N(t)\big) \psi\big(x_i^N(t),v_i^N(t)\big).
\end{align*}
Since $\theta$ is symmetric, hence $\theta^\ep\big(x_i^N(t)-x_k^N(t)\big)=\theta^\ep\big(x_k^N(t)-x_i^N(t)\big)$, after summing over all distinct pairs $(i,k)$ of the above six terms, we can switch the indices of $i$ and $k$ and regroup them into the claimed three sums.

The martingale associated with diffusion is
\begin{align*}
\wt M_T^D:=\frac{1}{N^2}\int_0^T\sum_{i\in\cN(t)}\sum_{j:j\neq i}(\nabla_vH^{\ep,z} \phi\psi+H^{\ep,z}\psi\nabla_v\phi )\, dB_i(t)+(-\nabla_vH^{\ep,z} \phi\psi+H^{\ep,z}\phi\nabla_v\psi) \, dB_j(t),
\end{align*}
hence
\begin{align}
\frac{1}{2}\E\big[|\wt M_T^D|^2\big]\le
&\E\frac{1}{N^4}\int_0^T\sum_{i\in\cN(t)}\Big|\sum_{j:j\neq i}(\nabla_vH^{\ep,z} \phi\psi+H^{\ep,z}\psi\nabla_v\phi)\Big|^2dt \label{mg-1}\\
&+\E\frac{1}{N^4}\int_0^T\sum_{j\in\cN(t)}\Big|\sum_{i:i\neq j}(-\nabla_vH^{\ep,z} \phi\psi+H^{\ep,z}\phi\nabla_v\psi)\Big|^2dt.\label{mg-2}
\end{align}
The martingale associated with jumps can be bounded via
\begin{align}
&\frac{1}{4}\E\big[|\wt M_T^J|^2\big]\nonumber\\
&\le\E\int_0^T\sum_{i\in\cN(t)}\sum_{k:k\neq i}\frac{1}{N}\theta^\ep\big(x_i^N(t)-x_k^N(t)\big)\Big[\frac{1}{N^2}\sum_{j:j\neq i,k}H^{\ep,z}\big(x_i^N(t)-x_j^N(t), v_i^N(t)-v_j^N(t)\big)\phi \psi\Big]^2dt\label{mg-triple-1}\\
&+\E\int_0^T\sum_{i\in\cN(t)}\sum_{k:k\neq i}\frac{1}{N}\theta^\ep\big(x_i^N(t)-x_k^N(t)\big)\Big[\frac{1}{N^2}\sum_{j:j\neq i,k}H^{\ep,z}\big(x_j^N(t)-x_i^N(t), v_j^N(t)-v_i^N(t)\big)\phi \psi\Big]^2dt\label{mg-triple-2}\\
&+\E\int_0^T\sum_{i\in\cN(t)}\sum_{k:k\neq i}\frac{1}{N}\theta^\ep\big(x_i^N(t)-x_k^N(t)\big)\Big[\frac{1}{N^2}H^{\ep,z}\big(x_i^N(t)-x_k^N(t), v_i^N(t)-v_k^N(t)\big)\phi \psi\Big]^2dt,\label{mg-triple-3}
\end{align}
cf. \cite[Proposition 8.7]{DN}.

The only term with second derivatives of $H^{\ep,z}$ from the above development is \eqref{2nd-order-2}. In order to substitute the highest derivative (also together with \eqref{2nd-order-1}), we set up the following subelliptic equation for $U^\ep:\R^{2d}\to\R$:
\begin{align}\label{ito-tanaka}
v\cdot \nabla_x U^\ep(x,v)+\Delta_v U^\ep(x,v) = -\theta^\ep(x)\chi^K(v).
\end{align}
We will choose in \eqref{gr-conv} a particular solution of \eqref{ito-tanaka}, which is the convolution of the \abbr{RHS} with the subelliptic Green's function. Recall in \eqref{def:H} we have defined
\[
H^{\ep,z}(x,v):=U^\ep(x+z,v)-U^\ep(x,v).
\]
\eqref{ito-tanaka} yields, by linearity, the equation satisfied by $H^{\ep,z}:\R^{2d}\to\R$, for any fixed $z$:
\begin{align}\label{ito-tanaka-2}
v\cdot \nabla_x H^{\ep,z}(x,v)+\Delta_v H^{\ep,z}(x,v) = \big[\theta^\ep(x)-\theta^\ep(x+z)\big]\chi^K(v).
\end{align}
Consequently, we can substitute \eqref{2nd-order-1}-\eqref{2nd-order-2} as follows
\begin{align*}
\eqref{2nd-order-1}+\eqref{2nd-order-2}=&\int_0^T\frac{1}{N^2}\sum_{i\in\cN(t)}\sum_{j:j\neq i}\\
&\cdot\big[\theta^\ep\big(x_i^N(t)-x_j^N(t)\big)-\theta^\ep\big(x_i^N(t)-x_j^N(t)+z\big)\big]\chi^K\big(v_i^N(t)-v_j^N(t)\big)\phi\big(x_i^N(t),v_i^N(t)\big) \psi\big(x_j^N(t),v_j^N(t)\big).
\end{align*}
The \abbr{RHS} is precisely the expression appearing in Proposition \ref{ppn:shift}. 

We remark at this stage that $\nabla_xH^{\ep,z}$ does not appear any more in the rest of terms coming from the It\^o's formula applied to $\cY_t$, after it is absorbed into the operator on the \abbr{LHS} of \eqref{ito-tanaka-2}. Its estimate is expected to be bad due to the degeneracy of the subellitic operator in the direction of $x$. What do appear in the rest of the terms are $\nabla_vH^{\ep,z}$ and $H^{\ep,z}$ itself, whose estimates are nice enough. We will show in the next sections that all the other terms \eqref{ini-term}, \eqref{double-sum-0}-\eqref{mg-terms} (i.e. apart from \eqref{2nd-order-1}-\eqref{2nd-order-2}) are negligible in $L^1(\Omega)$ in the order of limits $N\to\infty$ followed by $|z|\to0$, as required for Proposition \ref{ppn:shift}.

We can thus summarise the use of It\^o-Tanaka trick \eqref{ito-tanaka} in our context as follows. It allows to transfer proving a continuity claim Proposition \ref{ppn:shift} involving a nearly singular function $\theta^\ep$ (for $\ep$ small), into proving similar claims involving more regular functions $\nabla_vH^{\ep,z}$ and $H^{\ep,z}$. The estimates for the latter can be equivalently obtained from estimates for $U^\ep$, which is a convolution of $\theta^\ep(x)\chi^K(v)$ with a subelliptic Green's function. This convolution is a sort of regularization. Indeed, in the limit $\ep\to0$, $U^\ep$ does not tend to a singular object, but rather to a pointwise evaluation (in $x$) of the Green's function (suitably integrated in $v$). This is the reason why pointwise estimates and gradient estimates (in $v$) of the associated Green's function are key to our analysis.

\section{The hypoelliptic Green's function}\label{sec:hypoe}
Denote the operator 
\begin{align}\label{hyp-op}
L := v\cdot\nabla_x+\Delta_v
\end{align}
acting on appropriate functions on $\R^{2d}$, and we rewrite the \abbr{PDE} \eqref{ito-tanaka} more compactly as 
\begin{align}\label{aux-pde}
LU^\ep = -\theta^\ep(x)\chi^K(v).
\end{align}
It is well-known that $L$ \eqref{hyp-op} is an hypoelliptic operator in the sense of H\"ormander \cite{LH}, and its parabolic version (i.e. $\partial_t-L$) was earlier studied by Kolmogorov \cite{Ko} who wrote down its kernel explicitly. cf. \cite{AP} for a recent survey. Crucial for our purpose are pointwise estimates for the Green's function of $-L$.

To briefly recall that $L$ is hypoelliptic, let us write 
\begin{align}\label{degen-op}
L=A_0+\sum_{i=1}^d A_i^2
\end{align}
where 
\begin{align*}
A_i = \partial_{v_i}, \quad i=1,...,d; \quad A_0=\sum_{j=1}^dv_j\partial_{x_j}.
\end{align*}
Since the commutators
\begin{align*}
[A_0,A_i]=\sum_{j=1}^dv_j\partial_{x_j}(\partial_{v_i})-\partial_{v_i}(\sum_{j=1}^dv_j\partial_{x_j})=-\partial_{x_i}, \quad i=1,...,d
\end{align*}
we see that at every point in $\R^{2d}$,
\begin{align*}
\text{span}\big\{A_i, [A_0,A_i], \; i=1,...,d\big\}=\R^{2d}.
\end{align*}
Thus $L$ satisfies H\"ormander's condition \cite[Theorem 1.1]{LH}, \cite[pp. 128]{Nua}, implying that it is an hypoelliptic operator; that is, if $u$ is a distribution such that $Lu$ is $C^\infty$, then $u$ must be $C^\infty$ as well, cf. \cite[pp. 147]{LH}, \cite[pp. 129]{Nua}. 
Since the \abbr{RHS} of \eqref{aux-pde} is of class $C^\infty$ with compact support (for every fixed $\ep>0$), we can find a solution of the form 
\begin{align}\label{gr-conv}
U^\ep(x,v) = \iint_{\R^{2d}}G\big((x,v),(y,w)\big)\theta^\ep(y)\chi^K(w)dydw,
\end{align}
where $G=G_L$ is a distribution of class $C^\infty$ away from the diagonal, called the Green's function of $-L$, see e.g. \cite[pp. 143]{SC}, \cite[pp. 247]{FSC}. The definition of hypoellipticity implies that $U^\ep$ is of class $C^\infty$ for every fixed $\ep$; in particular, this is the function that is used in Section \ref{sec:ito} to define $\cY_t$.

To state the classical results of \cite{NSW,SC, FSC} on the estimates for the Green's function of general hypoelliptic operators, we recall that they are expressed in terms of a control metric (or Carnot-Carath\'eodory metric) $\rho_L(X,Y)$ canonically attached to $L$ (to avoid confusion, we use capital letters $X=(x,v),Y=(y,w),...$ to denote variables in $\R^{2d}$). We also set here
\begin{align*}
D:=2d.
\end{align*}
Let (in our case) $q=1+2d$ and $\{V_j,j=1,...,q\}$ are an enumeration of 
\begin{align*}
V_1:=A_0, \; V_{i+1}:=A_i, \; V_{i+d+1}:= [A_0,A_i], \quad i=1,...,d.
\end{align*}
As explained in \cite[pp.107]{NSW}, \cite[Section 18]{RS}, for operators of the form \eqref{degen-op} which are not just a sum-of-squares but with $A_0$, the degree of $A_0$ counts $2$ while the degree of each $A_i$ counts $1$, and likewise for their commutators. Hence, in our case we have that 
\begin{align*}
d_1=2, \; d_{i+1}=1,\; d_{i+d+1}=3, \quad i=1,...,d,
\end{align*}
where $d_j$ is the degree of $V_j$. 
\begin{definition}[{\cite[Definition 1.1]{NSW}}]\label{def:control}
Let $C(\delta)$ denote the class of absolutely continuous mappings $\phi:[0,1]\to\R^D$ which almost everywhere satisfy the differential equation
\begin{align}\label{control-velocity}
\phi'(t) = \sum_{j=1}^q a_j(t)V_j(\phi(t))
\end{align}
with $|a_j(t)|\le \delta^{d_j}$. Then define
\begin{align*}
\rho_L(X,Y):=\inf\big\{\delta>0:\;\exists \phi\in C(\delta), \phi(0)=X, \phi(1)=Y\big\}.
\end{align*}
\end{definition}
\begin{remark}
The geometric meaning of this definition is as follows. The ball of radius $\delta$ around the point $X$ in the control metric $\rho_L$ consists of those points $Y$ that can be reached at time $t=1$ by a ``particle'' (whose trajectory is denoted $\phi(t)$, $t\in[0,1]$), that starts at $t=0$ from $X$  and travels with velocity the \abbr{RHS} of \eqref{control-velocity}. The velocity at time $t$ has direction controlled by the vector fields $V_j$ evaluated at ``particle'' position, and modulus controlled together by the strength of the vector fields and the scalar multipliers $a_j(t)$.
\end{remark}
It is known that $\rho_L$ is a metric and in our case, it compares with the Euclidean distance $|\cdot|$ in the following way \cite[Proposition 1.1]{NSW}: for any compact set $\K\subset\R^D$, there are constants $C_i=C_i(\K,d)$, $i=1,2$, such that for any $X,Y\in\K$ we have that
\begin{align}\label{metric-comp}
C_1|X-Y|\le \rho_L(X,Y)\le C_2|X-Y|^{1/3}.
\end{align}
Given the control metric $\rho_L$ we define balls in this metric of radius $\delta>0$ by 
\begin{align*}
\B_L(X,\delta):=\big\{Y\in\R^D:\rho_L(X,Y)\le \delta\big\}.
\end{align*}
These balls are locally volume doubling as proved in \cite{NSW}.
\begin{theorem}[{\cite[Theorem 1, pp. 110]{NSW}}]
For every compact set $\K\subset\R^D$, there exist constants $C_1,C_2$ so that for all $X\in \K$, 
\begin{align*}
0<C_1\le \frac{|\B_L(X,\delta)|}{\Lambda(X,\delta)}\le C_2<\infty,
\end{align*}
where 
\begin{align*}
\Lambda(X,\delta):=\sum_{1\le i_1,...,i_D\le q}\big|\text{det}(V_{i_1},...,V_{i_D})(X)\big|\delta^{d_{i_1}+...+d_{i_D}}.
\end{align*}
\end{theorem}
Here, the differential operators $V_i$, $i=1,...,q$, are also viewed as vector fields, that is, 
\[
V_1(X)=(v_1,...,v_d,0,...0), \quad V_{i+1}(X)=e_{i+d}, \quad V_{i+d+1}(X)=-e_i, \quad i=1,...,d, \, X:=(x_1,...,x_d,v_1,...,v_d).
\]
where $\{e_i, i=1,...,D\}$ is the canonical basis of $\R^D$. Thus the above determinant for a $D\times D$ matrix.
In our case, it is straightforward to check that we have 
\begin{align}\label{volume}
\Lambda(X,\delta)=\sum_{i=1}^d|v_i|\delta^{4d-1}+\delta^{4d}, \quad X:=(x_1,...,x_d,v_1,...,v_d).
\end{align}
It is then clear that 
\begin{align}\label{vol-doub}
\frac{|\B_L(X,2\delta)|}{|\B_L(X,\delta)|}\le \frac{C_2\Lambda(X,2\delta)}{C_1\Lambda(X,\delta)}\le\frac{C_2}{C_1}2^{4d},
\end{align}
where the \abbr{RHS} is uniform for all $X\in\K$.

We can now state the pointwise estimates on Green's function and its partial derivatives in $v$, in the case of our operator $L$. As already remarked before, we do not need estimates for partial derivatives in $x$, which are worse.
\begin{theorem}[{\cite[Corollary, pp. 114]{NSW}, see also \cite[Theorem 1]{FSC}}]
Suppose that $D=2d\ge2$. For any compact set $\K\subset\R^D$, there exist some finite constants $C_i=C_i(\K,d)$, $i=1,2$, such that for any $X=(x,v),Y=(y,w)\in\K$, we have that 
\begin{align}
|G(X,Y)|&\le C_1\frac{\rho_L(X,Y)^2}{|\B_L\big(X,\rho_L(X,Y)\big)|}, \label{gr-bd}\\
|\nabla_vG(X,Y)|&\le C_2\frac{\rho_L(X,Y)}{|\B_L\big(X,\rho_L(X,Y)\big)|} \label{gr-bd-grad}.
\end{align}
\end{theorem}

We check that \eqref{gr-bd} and \eqref{gr-bd-grad} are sufficiently integrable near its singularity.
\begin{proposition}\label{ppn:unif-int}
Suppose that $D=2d\ge2$. There exists some $\eta=\eta(d)>0$ such that for any compact set $\K\subset\R^D$, and any $X=(x,v)\in\K$, $\delta\in(0,\infty)$, there are finite constants $C_i(\delta,\K,d)$, $i=1,2$ such that
\begin{align}
\int_{\B_L(X,\delta)}|G(X,Y)|^{1+\eta}dY&\le C_1 \label{int-bd-gr}\\
\int_{\B_L(X,\delta)}|\nabla_vG(X,Y)|^{1+\eta}dY&\le C_2.\label{int-bd-gr-grad}
\end{align}
\end{proposition}
\begin{proof}
We prove \eqref{int-bd-gr-grad} while the proof of \eqref{int-bd-gr} is similar. For $X\in\K\subset\R^D$ compact and $Y\in\B_L(X,\delta)$ with $\delta\in(0,\infty)$, they both belong to some other compact set $\K'$ that depends only on $\K, \delta$. This fact can be inferred from the meaning of the control metric Definition \ref{def:control} in the case of our operator $L$. By \eqref{gr-bd-grad}, there exists some $C=C(\K',d)$ such that for any $\eta>0$,
\begin{align*}
\int_{\B_L(X,\delta)}|\nabla_vG(X,Y)|^{1+\eta}dY&\le C\int_{\B_L(X,\delta)}\frac{\rho_L(X,Y)^{1+\eta}}{|\B_L\big(X,\rho_L(X,Y)\big)|^{1+\eta}}dY\\
&=C\sum_{k=0}^\infty\int_{\B_L(X,\delta 2^{-k})\backslash\B_L(X,\delta2^{-k-1})}\frac{\rho_L(X,Y)^{1+\eta}}{|\B_L\big(X,\rho_L(X,Y)\big)|^{1+\eta}}dY\\
&\le C\sum_{k=0}^\infty \frac{(\delta2^{-k})^{1+\eta}}{|\B_L\big(X,\delta2^{-k-1}\big)|^{1+\eta}}|\B_L(X,\delta2^{-k})\backslash\B_L(X,\delta2^{-k-1})|\\
&\le C\sum_{k=0}^\infty \frac{(\delta2^{-k})^{1+\eta}}{|\B_L\big(X,\delta2^{-k-1}\big)|^\eta}\frac{|\B_L(X,\delta2^{-k}))|}{|\B_L\big(X,\delta2^{-k-1}\big)|}.
\end{align*}
By \eqref{volume}, 
\begin{align*}
|\B_L\big(X, \delta2^{-k-1}\big)|\ge C_1(\K)\Lambda(X,\delta2^{-k-1})\ge C_1(\delta2^{-(k+1)})^{4d}.
\end{align*}
Combined with the local volume doubling property \eqref{vol-doub} of $\B_L(\cdot)$, we conclude that 
\begin{align*}
\int_{\B_L(X,\delta)}|\nabla_vG(X,Y)|^{1+\eta}dY\le C'\sum_{k=0}^\infty 2^{-k(1+\eta-4d\eta)},
\end{align*}
for some finite constant $C'=C'(\K,d, \delta)$. Choosing $\eta>0$ small enough such that $1+\eta-4d\eta>1/2$, the sum on the \abbr{RHS} is convergent.
\end{proof}

\medskip
Another remark is that in the case of our operator $-L$ \eqref{hyp-op}, its Green's function is translation invariant in $x$ (but not in $v$). That is, we claim that 
\begin{align}\label{transl}
G\big((x,v),(y,w)\big) = G\big((0,v),(y-x,w)\big).
\end{align}
This can be seen starting from the explicit kernel $\mathsf P_t\big((x,v),(y,w)\big)$ associated with the parabolic operator $\partial_t-L$. Since it is probabilistically linked to Gaussian processes, it can be explicitly calculated that, cf. \cite{Ko}, \cite[Eq. (3.24)-(3.25)]{FFPV}, \cite[Eq. (1.4)]{GT}, 
\begin{align}\label{kolmogorov}
\mathsf P_t\big((x,v),(y,w)\big)&=\left(\frac{\sqrt{3}}{2\pi t^2}\right)^d\exp\Big\{-\frac{|v-w|^2}{t}-3\frac{(v-w)\cdot(x-y-tv)}{t^2}-3\frac{|x-y-tv|^2}{t^3}\Big\}\nonumber\\
&=\left(\frac{\sqrt{3}}{2\pi t^2}\right)^d\exp\Big\{-3\big|\frac{v-w}{2t^{1/2}}+\frac{x-y-tv}{t^{3/2}}\big|^2-\frac{|v-w|^2}{4t}\Big\}.
\end{align}
which depends on $x,y$ through $y-x$ only. Since the Green's function is the time-integral of the parabolic kernel (cf. \cite[pp. 109]{NV}, \cite[pp. 220]{BAG}, \cite[pp. 276]{St}):
\begin{align*}
G\big((x,v),(y,w)\big)=\int_0^\infty \mathsf P_t\big((x,v), (y,w)\big) \, dt,
\end{align*}
the same translation invariance with respect to $x$ transfers to $G$ and thus we have \eqref{transl}.

\section{The control of negligible terms}\label{sec:control}
We proceed to bound all the terms coming out of It\^o's formula applied to $\cY_t$ in Section \ref{sec:ito}, that are expected to be negligible, thereby completing the proof of Proposition \ref{ppn:shift}. By a slight abuse of notation, we define on the same probability space $(\Omega,\cF,\P)$
\begin{align}\label{free}
X_i(t):=(x_i(t),v_i(t)):=\left(x_i(0)+v_i(0)t+\int_0^tB_i(s)ds, \, v_i(0)+B_i(t)\right), \quad i\in\N, \; t\ge0,
\end{align}
a countable collection of independent particles, with the {\it{same}} Brownian motions $\{B_i(t)\}_{i=1}^\infty$ as \eqref{dynamics} and {\it{same}} initial condition for $(x_i(0),v_i(0))_{i=1}^\infty$ with density $f_0(x,v)$ satisfying Condition \ref{cond:ini}. The only difference with \eqref{dynamics} is that the particles in \eqref{free} are non-interacting (i.e. they do not annihilate), hence they are independent, each with infinite lifespan. Further, there is an obvious a.s. domination between \eqref{dynamics} and \eqref{free} such that whenever $(x_i^N(t),v_i^N(t))$ is active, its trajectory coincides with that of $(x_i(t),v_i(t))$. 

\begin{lemma}\label{lem:den}
Let $d\ge1$. For any $t\ge0$, the density of $X_1(t)-X_2(t)$ is uniformly bounded by $\Gamma$, the uniform bound on $f_0$ in Condition \ref{cond:ini}. Further, the density $\mathsf p(t,x,v)$ of one particle $X_1(t)$ satisfies for every $t\in[0,T]$, $x,v\in\R^d$, the bound
\begin{align}\label{exp-decay}
\mathsf p(t,x,v)\le  Ce^{-|v|/C}1_{\{|v|\ge C\}}+\Gamma1_{\{|v|< C\}},
\end{align}
for some constant $C=C(R,\Gamma,T,d)$, where $R$ is the maximal radius of the compact support of $f_0$ in Condition \ref{cond:ini}.
\end{lemma}
\begin{proof}
The Markov generator of a single-particle process $X_1(t)$ is 
\[
L:=v\cdot\nabla_x+\frac{1}{2}\Delta_v.
\]
(Here we abuse notation and use the same letter $L$ even though the operator \eqref{hyp-op} discussed in Section \ref{sec:hypoe} differs by $1/2$ in front of the Laplacian.) The hypoellipticity of $L$ implies that $X_1(t)$, $t>0$, has a smooth density $\mathsf p(t,x,v)$ with respect to Lebesgue measure of $\R^{2d}$ \cite[Theorems 2.3.2, 2.3.3]{Nua}, and it is the unique solution of the Fokker-Planck equation
\begin{align}\label{evol-den}
\partial_t\mathsf p(t,x,v) &= L^*\mathsf p(t,x,v)\\
\mathsf p|_{t=0}&=f_0(x,v), \nonumber
\end{align}
where $L^*$ is the adjoint of $L$
\begin{align*}
L^*=-v\cdot\nabla_x+\frac{1}{2}\Delta_v.
\end{align*}
One can write $\mathsf p(t,x,v)$ explicitly as 
\[
\mathsf p(t,x,v)=\iint f_0(x_0,v_0)\mathsf P^\star_t\big((x_0,v_0), (x,v)\big)dx_0dv_0,
\]
where the kernel $\mathsf P^\star_t$ is given by a formula similar to \eqref{kolmogorov} but with certain sign and constants changes:
\begin{align}\label{one-particle}
\mathsf P^\star_t\big((x_0,v_0),(x,v)\big):=\left(\frac{\sqrt{3}}{\pi t^2}\right)^d\exp\Big\{-6\big|\frac{v-v_0}{2t^{1/2}}+\frac{x-x_0-tv_0}{t^{3/2}}\big|^2-\frac{|v-v_0|^2}{2t}\Big\}.
\end{align}
It then follows (e.g. from probabilistic representation of the solution of \eqref{evol-den}, cf. \cite[pp. 254, Eq. (21)]{Pr}, \cite[Lemma 2.4 (ii)]{GT}) that 
\begin{align*}
\left\|\mathsf p(t,\cdot)\right\|_{L^\infty}\le \|f_0\|_{L^\infty}\le\Gamma.
\end{align*}
Since the density of $X_1(t)-X_2(t)$ is a convolution between the respective densities of $X_1(t)$ and $-X_2(t)$, the two being independent, we conclude that it is also uniformly bounded by $\Gamma$.

Turning to a finer bound on $\mathsf p(t,x,v)$ when $|v|$ is large compared to $T$ and $R$, we note that by \eqref{one-particle},
\begin{align*}
\mathsf p(t,x,v)&=\iint_{\R^{2d}}f_0(x_0,v_0)\mathsf P^\star_t\big((x_0,v_0),(x,v)\big)dx_0dv_0\\
&\le C(\Gamma) \iint_{|x_0|,|v_0|\le R}\frac{1}{t^{2d}}\exp\left\{-\frac{|v-v_0|^2}{4t}\right\}dx_0dv_0.
\end{align*}
If $|v_0|\le R$ and 
\begin{align*}
|v|\ge 2R+T\ge t,
\end{align*}
then $|v-v_0|\ge |v|/2$ and $\lambda:=|v|^2/t \ge |v|$. Hence, for any such $v$ we have that 
\begin{align*}
\mathsf p(t,x,v)&\le C(\Gamma)R^{2d}\frac{|v|^{4d}}{t^{2d}}e^{-\frac{|v|^2}{16t}}|v|^{-4d}\\
& \le C(\Gamma, R, T,d)\lambda^{2d} e^{-\lambda/16}.
\end{align*}
Since $\lambda\mapsto \lambda^{2d}e^{-\lambda/16}$ is decreasing for $\lambda\ge |v|\ge \lambda_0$ for some $\lambda_0(d)<\infty$, we get the bound 
\begin{align*}
\mathsf p(t,x,v)\le C|v|^{2d} e^{-|v|/C},
\end{align*}
for any $|v|\ge C$, $x\in\R^d$ and $t\in[0,T]$, where $C=C(\Gamma, R, T,d)$ is a finite constant. Further adjusting the constant $C$ to absorb the polynomial factor, we get \eqref{exp-decay}.
\end{proof}

\medskip
The following proposition bounds \eqref{double-sum-1}. We note at this stage that since the independent system \eqref{free} dominates our true system \eqref{dynamics} in terms of lifespan, it suffices to prove all the following propositions for \eqref{free}.
\begin{proposition}\label{ppn:double-1}
For any finite $T$ and $d\ge1$, we have that
\begin{align*}
\lim_{|z|\to0}\limsup_{N\to\infty}\E\int_0^Tdt\frac{1}{N^2}\sum_{i\neq j=1}^N\Big|v_i(t)\cdot\nabla_x\phi\big(t, x_i(t),v_i(t)\big) \psi\big(x_j(t),v_j(t)\big) H^{\ep,z}\big(x_i(t)-x_j(t), v_i(t)-v_j(t)\big)\Big|=0.
\end{align*}
\end{proposition}
\begin{proof}
Recall that we have denoted by $K=K^{\phi,\psi}$ the maximal radius of the compact supports of $\phi, \psi$, and defined a compactly-supported function $\chi^K:\R^d\to[0,1]$ \eqref{compact}. For the quantity in the thesis, it suffices to consider 
\begin{align*}
|x_i(t)|, |x_j(t)|, |v_i(t)|, |v_j(t)|\le K
\end{align*}
(otherwise one of the arguments of $\phi$ and $\psi$ is zero). When this is the case, we have that 
\begin{align*}
|x_i(t)-x_j(t)|, |v_i(t)-v_j(t)|\le 2K.
\end{align*}
We control each term of the double sum in the thesis, and by exchangeability it suffices to focus on the term containing the particles with index $1$ and $2$:
\begin{align*}
&\E\int_0^Tdt\Big|v_1(t)\cdot\nabla_x\phi\big(t, x_1(t),v_1(t)\big) \psi\big(x_2(t),v_2(t)\big)H^{\ep,z}\big(x_1(t)-x_2(t), v_1(t)-v_2(t)\big)\Big|\\
&\le K\|\phi\|_{C^1}\|\psi\|_{L^\infty}\E\int_0^Tdt\big|H^{\ep,z}\big(x_1(t)-x_2(t), v_1(t)-v_2(t)\big)\big|\chi^K\big(x_1(t)-x_2(t)\big)\chi^K\big(v_1(t)-v_2(t)\big).
\end{align*}
Denote by $\mathsf p^{12}(t,x,v)$ the density of $\big(x_1(t)-x_2(t), v_1(t)-v_2(t)\big)$, which by Lemma \ref{lem:den} is uniformly bounded by $\Gamma$. We can write the expectation above explicitly as 
\begin{align}\label{eq:main-1}
\int_0^Tdt\iint_{\R^{2d}}|H^{\ep,z}(x,v)|\chi^K(x)\chi^K(v)\mathsf p^{12}(t,x,v)dxdv.
\end{align}
By \eqref{gr-conv}, \eqref{transl}, a change of variables $\wt y=y-x$ and the symmetry of $\theta^\ep$, we have that 
\begin{align}\label{change}
H^{\ep,z}(x,v)&=\iint_{\R^{2d}}G\big((x,v),(y,w)\big)\big[\theta^\ep(y)-\theta^\ep(y+z)\big]\chi^K(w)dydw\nonumber\\
&=\iint_{\R^{2d}}G\big((0,v),(y-x,w)\big)\big[\theta^\ep(y)-\theta^\ep(y+z)\big]\chi^K(w)dydw\nonumber\\
&=\iint_{\R^{2d}}G\big((0,v),(\wt y,w)\big)\big[\theta^\ep(\wt y+x)-\theta^\ep(\wt y+x+z)\big]\chi^K(w)d\wt ydw\nonumber\\
&=\iint_{\R^{2d}}G\big((0,v),(\wt y,w)\big)\big[\theta^\ep(-\wt y-x)-\theta^\ep(-\wt y-x-z)\big]\chi^K(w)d\wt ydw\nonumber\\
&=\iint_{\R^{2d}}\big[G\big((0,v),(\wt y,w)\big)-G\big((0,v),(\wt y-z,w)\big)\big]\theta^\ep(-\wt y-x)\chi^K(w)d\wt ydw.
\end{align}
It yields that 
\begin{align*}
&\iint_{\R^{2d}}|H^{\ep,z}(x,v)|\chi^K(x)\chi^K(v)\mathsf p^{12}(t,x,v)dxdv\\
&\le \iiint_{\R^{3d}}\big|G\big((0,v),(\wt y,w)\big)-G\big((0,v),(\wt y-z,w)\big)\big|\chi^K(w)\Big(\int_{\R^d}\theta^\ep(-\wt y-x)\chi^K(x)\chi^K(v)\mathsf p^{12}(t,x,v)dx\Big)d\wt ydwdv\\
&\overset{\ep\to0}{\to}\iiint_{\R^{3d}}\big|G\big((0,v),(\wt y,w)\big)-G\big((0,v),(\wt y-z,w)\big)\big|\chi^K(w)\chi^K(-\wt y)\chi^K(v)\mathsf p^{12}(t,-\wt y,v)d\wt ydwdv.
\end{align*}
Since the \abbr{RHS} no longer depends on $\ep$, and also $\mathsf p^{12}$ is bounded uniformly by $\Gamma$, it suffices to prove that the following converges to $0$ as $|z|\to0$:
\begin{align*}
\iiint_{\R^{3d}}\big|G\big((0,v),(\wt y,w)\big)-G\big((0,v),(\wt y-z,w)\big)\big|\chi^K(w)\chi^K(-\wt y)\chi^K(v)d\wt ydwdv.
\end{align*}
By the continuity of $G$, the integrand converges to $0$ as $|z|\to0$ Lebesgue a.e. Further, due to the compact support of $w,\wt y,v$ we can apply the finite measure case of Vitali's convergence theorem, if we can check uniform integrability. By the de la Vall\'ee-Poussin theorem it is enough to check that there exists some $\eta=\eta(d)>0$ and $C=C(K,d)$ finite such that
\begin{align}\label{vitali}
\sup_{|z|\le 1}\iiint_{\R^{3d}}\big|G\big((0,v),(\wt y-z,w)\big)\big|^{1+\eta}\chi^K(w)\chi^K(-\wt y)\chi^K(v)d\wt ydwdv<C.
\end{align}
Note that the support of $(0,v)$ is in a compact set, the Euclidean ball $\B_E((0,0),2K)$ in $\R^{2d}$. By \eqref{metric-comp} there exists a ball in the metric $\rho_L$ of some radius $K'$ (depending only on $K$ and $d$) such that for any $|z|\le1$ and $-\wt y,w$ in the compact support of $\chi^K$,
\begin{align*}
(\wt y-z,w)\in\B_E\big((0,v), 16K\big)\cap\B_E\big((0,0),32K\big)\subset \B_L\big((0,v),K'\big).
\end{align*}
Hence, we are in the setting of Proposition \ref{ppn:unif-int}. By \eqref{int-bd-gr} there exists some $\eta=\eta(d)>0$ and constant $C=C(K,K',d)$ finite such that for any fixed $v$ with $|v|\le 2K$ we have that 
\begin{align*}
\iint_{\B_L\left((0,v),K'\right)}\big|G\big((0,v),(\wt y-z,w)\big)\big|^{1+\eta}d\wt ydw \le C.
\end{align*}
Integrating further in $|v|\le 2K$ yields a different constant which proves \eqref{vitali}, hence the whole thesis.
\end{proof}
\begin{remark}
An analogous proof as in Proposition \ref{ppn:double-1} can treat \eqref{ini-term}, \eqref{double-sum-0} and \eqref{double-sum-2} and prove their negligibility. Indeed, they are all of the form of an averaged (by $N^{-2}$) double sum over distinct pairs of particles $i\neq j$, of an expression that consists of $H^{\ep,z}(x_i-x_j, v_i-v_j)$ multipled by two compact-supported smooth test functions localizing all four variables $x_i,x_j,v_i,v_j$. Since it is an averaged double sum, we can focus on controlling the term of a particular pair $i=1, j=2$. Despite the test functions slightly differ, we only use that they are bounded and that they make the variables compactly supported. Finally, the uniform boundedness of the particle joint densities $\mathsf p^{12}$ suggests that the time-integral $\int_0^T$ is not important, and the same proof works for the terms \eqref{ini-term}, corresponding to the process $\cY_t$ at initial and terminal times.
\end{remark}

\medskip
To prepare for the next proposition, we require a crude pointwise bound on $H^{\ep,z}$ and its partial derivative in $v$. It suffices to derive such for the function $U^\ep$.

\begin{lemma}\label{crude}
Suppose that $D=2d\ge2$. For any compact set $\K\subset\R^D$, there exists some finite constant $C=C(\K,K,\theta,d)$ such that for any $(x,v)\in\K$ we have that
\begin{align}\label{unif-U}
|U^\ep(x,v)|, |\nabla_vU^\ep(x,v)|\le C\ep^{-d}.
\end{align}
\end{lemma}
\begin{proof}
We prove the thesis for the gradient, the other part being similar. Recall \eqref{gr-conv}, we have that 
\begin{align*}
\nabla_vU^\ep(x,v) = \int_{\R^{2d}}\nabla_vG\big((x,v),(y,w)\big)\theta^\ep(y)\chi^K(w)dydw.
\end{align*}
Note that, with the function $\chi^1$ defined in \eqref{compact}, we can write
\begin{align*}
\theta^\ep(y) \le \ep^{-d}\|\theta\|_{L^\infty}\chi^1(y).
\end{align*}
Hence, by \eqref{int-bd-gr-grad} (with $\eta=0$), there exists some finite constant $C=C(\K,K,\theta,d)$ such that for any $(x,v)\in\K$, we have
\begin{align*}
|\nabla_vU^\ep(x,v)|\le\ep^{-d}\|\theta\|_{L^\infty}\int_{\R^{2d}}|\nabla_vG\big((x,v),(y,w)\big)|\chi^1(y)\chi^K(w)dydw\le \ep^{-d}C.
\end{align*}
Indeed, the range of integration of $(y,w)$ is in $\B_E((0,0),8K)$, which is inside $\B_L((x,v), K')$ for some $K'<\infty$ depending on $\K,K$, uniformly for all $(x,v)\in\K$.
\end{proof}

\medskip
The following proposition bounds \eqref{mg-1}.  An analogous proof can treat \eqref{mg-2} and prove its negligibility. Indeed, note that the only difference between \eqref{mg-1} and \eqref{mg-2} are their test functions, i.e. instead of $\psi$ we have $\nabla_v\psi$ or instead of $\nabla_v\phi$ we have $\phi$. The proof only used that they are smooth bounded and  they localize the variables $x_i,x_j, v_i,v_j$ in compact sets.
\begin{proposition}
For any finite $T$ and $d\ge1$, we have that 
\begin{align*}
&\lim_{|z|\to0}\limsup_{N\to\infty}\E\int_0^Tdt\frac{1}{N^4}\sum_{i=1}^N\Big|\sum_{j\neq i}\nabla_vH^{\ep,z} \big(x_i(t)-x_j(t), v_i(t)-v_j(t)\big)\phi\big(t, x_i(t),v_i(t)\big)\psi\big(x_j(t),v_j(t)\big)\\
&+\sum_{j\neq i}H^{\ep,z}\big(x_i(t)-x_j(t), v_i(t)-v_j(t)\big)\psi\big(x_j(t),v_j(t)\big)\nabla_v\phi\big(x_i(t),v_i(t)\big)\Big|^2=0.
\end{align*}
\end{proposition}

\begin{proof}
%The proof is very similar to that of Proposition \ref{ppn:double-1}. 
First, using the elementary inequality $|\sum_{k=1}^Ma_k|^2\le M\sum_{k=1}^M|a_k|^2$, it suffices to bound separately
\begin{align}\label{sub-item-1}
\E\int_0^Tdt\frac{1}{N^3}\sum_{i,j=1}^N\Big|\nabla_vH^{\ep,z} \big(x_i(t)-x_j(t), v_i(t)-v_j(t)\big)\phi\big(t, x_i(t),v_i(t)\big)\psi\big(x_j(t),v_j(t)\big)\Big|^2
\end{align}
and
\begin{align}\label{sub-item-2}
\E\int_0^Tdt\frac{1}{N^3}\sum_{i,j=1}^N\Big|H^{\ep,z}\big(x_i(t)-x_j(t), v_i(t)-v_j(t)\big)\psi\big(x_j(t),v_j(t)\big)\nabla_v\phi\big(t, x_i(t),v_i(t)\big)\Big|^2.
\end{align} 
With the arguments of $H^{\ep,z}$ in a compact set, we can bound one term of the square in \eqref{sub-item-2} by the crude bound \eqref{unif-U} in Lemma \ref{crude}, then using the relation $\ep^{-d}\le N$ in Condition \ref{cond:local}, we get that 
\begin{align*}
\eqref{sub-item-2}\le C(K,d, \theta)\|\phi\|_{C^1}\|\psi\|_{L^\infty} \E\int_0^Tdt\frac{1}{N^2}\sum_{i,j=1}^N\Big|H^{\ep,z}\big(x_i(t)-x_j(t), v_i(t)-v_j(t)\big)\psi\big(x_j(t),v_j(t)\big)\nabla_v\phi\big(t, x_i(t),v_i(t)\big)\Big|.
\end{align*}
This expression is negligible which can be proved with the same argument as Proposition \ref{ppn:double-1}, as it involves the same type of double sum. Only the text functions look different, but all we use of the test functions are their boundedness and that they localize the variable in a compact set.

Turning to \eqref{sub-item-1}, again by \eqref{unif-U} and Condition \ref{cond:local}, it is sufficient to bound
\begin{align*}
\E\int_0^Tdt\frac{1}{N^2}\sum_{i,j=1}^N\Big|\nabla_vH^{\ep,z}\big(x_i(t)-x_j(t), v_i(t)-v_j(t)\big)\psi\big(x_j(t),v_j(t)\big)\phi\big(t, x_i(t),v_i(t)\big)\Big|.
\end{align*}
The proof of the negligibility of the above quantity is analogous to the proof presented in Proposition \ref{ppn:double-1}, just using the bound \eqref{int-bd-gr-grad} for the gradient instead of \eqref{int-bd-gr}, hence we only sketch the idea. Indeed, as in \eqref{change}, we have that 
\begin{align*}
\nabla_vH^{\ep,z}(x,v)=\iint_{\R^{2d}}\big[\nabla_vG\big((0,v),(\wt y,w)\big)-\nabla_vG\big((0,v),(\wt y-z,w)\big)\big]\theta^\ep(-\wt y-x)\chi^K(w)d\wt ydw.
\end{align*}
One can now take $\ep\to0$ to obtain
\begin{align*}
&\limsup_{\ep\to0}\iint_{\R^{2d}}|\nabla_vH^{\ep,z}(x,v)|\chi^K(x)\chi^K(v)\mathsf p^{12}(t,x,v)dxdv\\
&\le\iiint_{\R^{3d}}\big|\nabla_vG\big((0,v),(\wt y,w)\big)-\nabla_vG\big((0,v),(\wt y-z,w)\big)\big|\chi^K(w)\chi^K(-\wt y)\chi^K(v)\mathsf p^{12}(t,-\wt y,v)d\wt ydwdv.
\end{align*}
In order to use Vitali convergence theorem to take $|z|\to0$, we need to check uniform integrability similar to \eqref{vitali}. In this case, the estimate \eqref{int-bd-gr-grad} guarantees the existence of some $\eta=\eta(d)>0$ and $C=C(K,d)$ finite such that 
\begin{align*}
\sup_{|z|\le 1}\iiint_{\R^{3d}}\big|\nabla_vG\big((0,v),(\wt y-z,w)\big)\big|^{1+\eta}\chi^K(w)\chi^K(-\wt y)\chi^K(v)d\wt ydwdv<C.
\end{align*}
\end{proof}

\medskip
The following proposition bounds \eqref{triple-sum-1}. The same proof can treat \eqref{triple-sum-2}, indeed its structure is the same as \eqref{triple-sum-1} except the roles of $\phi$ and $\psi$ are permuted. Since they are both smooth test functions, this fact does not alter the proof in an essential way. For \eqref{triple-sum-3}, note that upon bounding $\|\theta^\ep\|_{L^\infty}\le\ep^{-d}\|\theta\|_{L^\infty}$ and using $\ep^{-d}\le N$, it reduces to Proposition \ref{ppn:double-1}.

\begin{proposition}\label{ppn:triple}
For any finite $T$ and $d\ge1$, we have that 
\begin{align*}
\lim_{|z|\to0}\limsup_{N\to\infty}\E\int_0^Tdt\frac{1}{N^3}\sum_{i\neq j,i\neq  k=1}^N&\theta^\ep\big(x_i(t)-x_k(t)\big)\\
&\cdot\Big|H^{\ep,z}\big(x_i(t)-x_j(t), v_i(t)-v_j(t)\big)\phi\big(t, x_i(t),v_i(t)\big) \psi\big(x_j(t),v_j(t)\big)\Big|=0.
\end{align*}
\end{proposition}
\begin{proof}
First, we note that for $j=k$, we can bound $\|\theta^\ep\|_{L^\infty}\le\ep^{-d}\|\theta\|_{L^\infty}$ and with the condition $\ep^{-d}\le N$, we have a contribution 
\begin{align*}
&\E\int_0^Tdt\frac{1}{N^3}\sum_{i\neq j=1}^N\theta^\ep\big(x_i(t)-x_j(t)\big)\Big|H^{\ep,z}\big(x_i(t)-x_j(t), v_i(t)-v_j(t)\big)\phi\big(t, x_i(t),v_i(t)\big) \psi\big(x_j(t),v_j(t)\big)\Big|\\
&\le\E\int_0^Tdt\frac{1}{N^2}\sum_{i\neq j=1}^N\Big|H^{\ep,z}\big(x_i(t)-x_j(t), v_i(t)-v_j(t)\big)\phi\big(t, x_i(t),v_i(t)\big) \psi\big(x_j(t),v_j(t)\big)\Big|,
\end{align*}
whose negligibility follows from Proposition \ref{ppn:double-1}. Hence, it suffices to consider the contribution of $i\neq j\neq k$ all distinct. By similar considerations as in the beginning of the proof of Proposition \ref{ppn:double-1}, it suffices to bound each term of the form
\begin{align*}
\E\int_0^Tdt\;\theta^\ep\big(x_1(t)-x_3(t)\big)\Big|H^{\ep,z}\big(x_1(t)-x_2(t), v_1(t)-v_2(t)\big)\Big|\chi^K(x_1(t))\chi^K(x_2(t))\chi^K(v_1(t))\chi^K(v_2(t)).
\end{align*}
By independence, the joint density of $(x_i(t),v_i(t))_{i=1}^3$ is given by $\prod_{i=1}^3\mathsf p(t,x_i,v_i)$. We write the above expectation explicitly 
\begin{align*}
\int_0^Tdt \int_{\R^{6d}}\theta^\ep(x_1-x_3)\Big|H^{\ep,z}\big(x_1-x_2, v_1-v_2\big)\Big|\chi^K(x_1))\chi^K(x_2)\chi^K(v_1)\chi^K(v_2)\prod_{i=1}^3\mathsf p(t,x_i,v_i)dx_idv_i.
\end{align*}
By \eqref{exp-decay} estimate for the one-particle density, we integrate first in $x_3,v_3$
\begin{align*}
&\int_{\R^{2d}}\theta^\ep(x_1-x_3)\mathsf p(t,x_3,v_3)dx_3dv_3\\
&\le \int_{\R^{2d}}\theta^\ep(x_1-x_3)\big[Ce^{-|v|/C}1_{\{|v|\ge C\}}+\Gamma1_{\{|v|< C\}}\big]dx_3dv_3 \le C'(\Gamma, R, T,d),
\end{align*}
where we used that $\int\theta^\ep(x_1-x_3)dx_3=1$ for any $x_1$. We are left to bound 
\begin{align*}
\int_0^Tdt \int_{\R^{4d}}\Big|H^{\ep,z}\big(x_1-x_2, v_1-v_2\big)\Big|\chi^K(x_1))\chi^K(x_2)\chi^K(v_1)\chi^K(v_2)\prod_{i=1}^2\mathsf p(t,x_i,v_i)dx_idv_i.
\end{align*}
Recall the definition of $\chi^K$ \eqref{compact} that $x_1,x_2,v_1,v_2\in\text{supp}(\chi^K)$ implies that $x_1-x_2, v_1-v_2\in\text{supp}(\chi^{4K})$, and $\chi^{4K}(x_1-x_2)=\chi^{4K}(v_1-v_2)=1$. Also $\mathsf p(t,x_i,v_i)\le\Gamma$, $i=1,2$, hence the preceding display can be bnounded above by
\begin{align*}
&T\Gamma^2\int_{\R^{4d}}\Big|H^{\ep,z}\big(x_1-x_2, v_1-v_2\big)\Big|\chi^{4K}(x_1-x_2)\chi^{4K}(v_1-v_2)\chi^K(x_1))\chi^K(x_2)\chi^K(v_1)\chi^K(v_2)dx_1dv_1dx_2dv_2\\
&\stackrel[\wt v=v_1-v_2]{\wt x=x_1-x_2}{=}T\Gamma^2\int_{\R^{4d}}\Big|H^{\ep,z}\big(\wt x, \wt v\big)\Big|\chi^{4K}(\wt x)\chi^{4K}(\wt v)\chi^K(\wt x+x_2))\chi^K(x_2)\chi^K(\wt v+v_2)\chi^K(v_2)d\wt xd\wt vdx_2dv_2\\
&\le T\Gamma^2\int_{\R^{4d}}\Big|H^{\ep,z}\big(\wt x, \wt v\big)\Big|\chi^{4K}(\wt x)\chi^{4K}(\wt v)\chi^K(x_2)\chi^K(v_2)d\wt xd\wt vdx_2dv_2\\
&\le C(K) T\Gamma^2\iint_{\R^{2d}}\Big|H^{\ep,z}\big(\wt x, \wt v\big)\Big|\chi^{4K}(\wt x)\chi^{4K}(\wt v)d\wt xd\wt v,
\end{align*}
where the second line is due to a change of variable, the third line due to $\chi^K\in[0,1]$ and the last line is because the integration in variables $x_2,v_2$ is over the compact support of $\chi^K$. 

The resulting expression can be treated in the same way as \eqref{eq:main-1} in the proof of Proposition \ref{ppn:double-1}, and its negligibility in the limit $N\to\infty$ followed by $|z|\to0$ follows therefrom.
\end{proof}

\medskip
Finally, we remark that the control of \eqref{mg-triple-1}, \eqref{mg-triple-2}, \eqref{mg-triple-3} reduces to Propositions \ref{ppn:triple} and  \ref{ppn:double-1}. Indeed, take \eqref{mg-triple-1} for instance, using the elementary inequality $|\sum_{k=1}^Ma_k|^2\le M\sum_{k=1}^M|a_k|^2$ and the bound \eqref{unif-U}, we have that 
\begin{align*}
\eqref{mg-triple-1}&=\E\int_0^Tdt\sum_{i\neq k=1}^N\frac{1}{N^5}\theta^\ep\big(x_i^N(t)-x_k^N(t)\big)\Big|\sum_{j:j\neq i,k}H^{\ep,z}\big(x_i^N(t)-x_j^N(t), v_i^N(t)-v_j^N(t)\big)\phi \psi\Big|^2\\
&\le \E\int_0^Tdt\sum_{i\neq j\neq  k=1}^N\frac{1}{N^4}\theta^\ep\big(x_i^N(t)-x_k^N(t)\big)\Big|H^{\ep,z}\big(x_i^N(t)-x_j^N(t), v_i^N(t)-v_j^N(t)\big)\phi \psi\Big|^2\\
&\le C\|\phi\|_{L^\infty}\|\psi\|_{L^\infty}\E\int_0^Tdt\sum_{i\neq j\neq  k=1}^N\frac{1}{N^3}\theta^\ep\big(x_i^N(t)-x_k^N(t)\big)\Big|H^{\ep,z}\big(x_i^N(t)-x_j^N(t), v_i^N(t)-v_j^N(t)\big)\phi \psi\Big|.
\end{align*}
It then reduces to Propositions \ref{ppn:triple}.

\medskip
We have controlled all the negligible terms \eqref{ini-term}, \eqref{double-sum-0}-\eqref{mg-terms}, and by the discussions in Section \ref{sec:ito}, we have established Proposition \ref{ppn:shift}.

\section{Tightness and subsequential convergence}\label{sec:tight}
We show that for any finite $T$, the laws of the sequence $\{\mu_t^N(dx,dv):t\in[0,T]\}_{N\in\N}$ \eqref{empirical} of $\cD_T(\cM_{+,1})$-valued random variables is tight, hence weakly relatively compact by Prokhorov's theorem, where $\cD_T(\cM_{+,1})=\cD([0,T];\cM_{+,1}(\R^{2d}))$ is the Skorohod space \eqref{skorohod}. The following is a sufficient condition.
\begin{proposition}\label{ppn:tight}
If (a). there exists constant finite $C_1$ such that for any $N\in\N$ and $t\in[0,T]$ we have that 
\begin{align}\label{AA-1}
\E\int_{\R^{2d}}\big(|x|+|v|\big)\mu_t^N(dx,dv)\le C_1;
\end{align}
and (b). there exists some $\beta>0$ such that for every $\phi\in C_c^\infty(\R^{2d})$ there exists $C(\phi)<\infty$ such that for all $N$ and all stopping times $\tau\le T$, $\gamma>0$ small enough and $0<\sigma<\gamma$,
\begin{align}\label{AA-2}
\E\big|\langle \mu_{\tau+\sigma}^N,\, \phi\rangle - \langle \mu_\tau^N,\, \phi\rangle\big|\le C(\phi)\,\gamma^\beta;
\end{align}
then the laws $\{\mathsf P_N\}_N$ of the sequence $\{\mu_t^N:t\in[0,T]\}_{N\in\N}$ is tight.
\end{proposition}

\begin{proof}
Tightness of a collection of laws on a Skorohod space $\cD([0,T]; E)$, where $E$ is a complete separable metric space, is well-studied in the literature; it builds on the characterization of compact sets in these spaces, e.g.  \cite[Theorem 3.12.3]{PB}, \cite[Theorem 3.6.3]{EK}, \cite[Proposition 4.1.2]{KL}. 

Specializing to our case, $E=\cM_{+,1}(\R^{2d})$, endowed with the following metric that metrizes weak convergence, cf. \cite[pp. 49, Eq. (1.1)]{KL}, \cite[Section 5.1]{FH2}
\begin{align}\label{metric-weak}
\delta(\mu, \nu):=\sum_{k=1}^\infty \frac{1}{2^k}\left[\left|\langle\mu, \phi_k\rangle- \langle\nu, \phi_k\rangle\right|\wedge 1\right],\quad \mu, \nu \in \cM_{+.1}(\R^{2d}),
\end{align}
where $\{\phi_k\}_{k=1}^\infty$ is a fixed dense sequence in $C_c^\infty(\R^{2d})$.
Recall \cite[Theorem 4.1.3, Proposition 4.1.6]{KL} that a sequence of probability law $\{\mathsf P_N\}_N$ on $\cD_T(\cM_{+,1})$ is relatively compact if the following two conditions hold: (a'). For every $t>0$ and $\iota>0$, there exists a compact set $K(t,\iota)\subset\cM_{+,1}$ such that 
\begin{align}\label{KL-1}
\sup_N\mathsf P_N\left(\mu_t\not\in K(t,\iota)\right)<\iota.
\end{align}
(b'). For every $\iota>0$, 
\begin{align}\label{KL-2}
\lim_{\gamma\to0}\limsup_{N\to\infty}\sup_{\tau\in\mathcal T_T, \, \sigma\le \gamma}\mathsf P_N\left(\delta(\mu_\tau, \mu_{\tau+\sigma})>\iota\right)=0,
\end{align}
where $\mathcal T_T$ is the set of stopping times $\le T$.

Now recall that given $R<\infty$, the set
\[
\cK^{R}:=\left\{\mu\in \cM_{+,1}(\R^{2d}):\, \int_{\R^{2d}}\left(|x|+|v|\right)\mu(dx,dv)\le R\right\}
\]
is relatively compact in $\cM_{+,1}(\R^{2d})$, since it consists of a tight collection of subprobability measures on $\R^{2d}$. Indeed, for any $A>0$ we have that 
\[
\sup_{\mu\in \cK^{R}}\mu\left((x,v)\in\R^{2d}: |x|+|v|> A\right)\le \sup_{\mu\in \cK^{R}}\frac{1}{A}\int_{\R^{2d}}\left(|x|+|v|\right)\mu(dx,dv)\le \frac{R}{A}.
\]
Therefore, condition \eqref{AA-1} implies that 
\[
\mathsf P_N\left(\mu_t\not\in \ovl{\cK^R}\right)\le \P\left(\int_{\R^{2d}}\left(|x|+|v|\right)\mu(dx,dv)>R\right)\le \frac{1}{R}\E\int_{\R^{2d}}\left(|x|+|v|\right)\mu(dx,dv)\le \frac{C_1}{R},
\]
by Markov's inequality, and the \abbr{RHS} can be made as small as one wants, uniformly in $N$ (and $t$) upon choosing $R$ large enough. This yields condition \eqref{KL-1}.

Turning to consider \eqref{AA-2}, let us denote $c_k=C(\phi_k)$ the constant on the \abbr{RHS} of \eqref{AA-2}. Given any $\iota>0$, let $k_0=k_0(\iota)$ be the largest integer $k'$ such that $\sum_{k=k'}^\infty 2^{-k}\ge \iota/2$. By the definition of \eqref{metric-weak}, we have that 
\begin{align*}
\delta\left(\mu_\tau, \mu_{\tau+\sigma}\right)\le \sum_{k=1}^{k_0}\frac{1}{2^k}\left|\langle\mu_\tau, \phi_k\rangle- \langle\mu_{\tau+\sigma}, \phi_k\rangle\right|+\frac{\iota}{2}
\end{align*}
Hence, 
\begin{align*}
\mathsf P_N\left(\delta\left(\mu_\tau, \mu_{\tau+\sigma}\right)>\iota\right)&\le \mathsf P_N\left(\sum_{k=1}^{k_0}\frac{1}{2^k}\left|\langle\mu_\tau, \phi_k\rangle- \langle\mu_{\tau+\sigma}, \phi_k\rangle\right|>\frac{\iota}{2}\right)\\
&\le \sum_{k=1}^{k_0}\mathsf P_N\left(\left|\langle\mu_\tau, \phi_k\rangle- \langle\mu_{\tau+\sigma}, \phi_k\rangle\right|>\frac{\iota}{2}\right).
\end{align*}
By \eqref{AA-2} and Markov's inequality, for every $k\le k_0$ and $\tau\in\cT_t$, $\sigma\le\gamma$,
\[
\mathsf P_N\left(\left|\langle\mu_\tau- \phi_k\rangle-\langle\mu_{\tau+\sigma}, \phi_k\rangle\right|>\frac{\iota}{2}\right)\le \frac{2}{\iota}\E\left|\langle\mu^N_\tau, \phi_k\rangle-\langle\mu^N_{\tau+\sigma}, \phi_k\rangle\right|\le \frac{2c_k\gamma^\beta}{\iota}.
\]
Given any $\lambda>0$, there exists $\gamma_0=\gamma_0(\lambda,\iota, c_1,..., c_{k_0})>0$ such that $2c_k\gamma^\beta/\iota<\lambda/2^k$ for all $k\le k_0$ and $\gamma\in(0,\gamma_0)$. Thus, for all $N$ and $\tau\in\cT_T$, $\sigma\le \gamma$, $\gamma<\gamma_0$, we have that 
\[
\mathsf P_N\left(\delta\left(\mu_\tau, \mu_{\tau+\sigma}\right)>\iota\right)\le\sum_{k=1}^{k_0}\frac{\lambda}{2^k}< \lambda.
\]
Since $\lambda>0$ is arbitrary, we get \eqref{KL-2}. This completes the proof that \eqref{AA-1}-\eqref{AA-2} is a sufficient condition for the sequence of laws $\mathsf P_N$ to be tight.
\end{proof}

\medskip
Now we check the two conditions in our case. Despite that \eqref{AA-2} involves stopping times, with proper conditioning on the stopped filtration at $\tau$ and using the strong Markov property, it is enough to write the proof for deterministic times $s,t$.

For (a), by the exchangeability of independent particles $\{x_i(t), v_i(t)\}_{i=1}^N$ \eqref{free} whose trajectories coincide with the active particles in the true system $\{x_i^N(t)\}_{i=1}^{N(t)}$ \eqref{dynamics}, we have that 
\begin{align*}
\E\int_{\R^{2d}}\big(|x|+|v|\big)\mu_t^N(dx,dv)&=\frac{1}{N}\sum_{i=1}^{N(t)}\E\big(|x^N_i(t)|+|v_i^N(t)|\big)\\
&\le \frac{1}{N}\sum_{i=1}^{N}\E\big(|x_i(t)|+|v_i(t)|\big)=\E\big(|x_1(t)|+|v_1(t)|\big).
\end{align*}
Since 
\begin{align*}
(x_1(t),v_1(t))=\big(x_1(0)+v_1(0)t+\int_0^tB_1(s)ds, \, v_1(0)+B_1(t)\big)
\end{align*}
for any $t\in[0,T]$ and $(x_1(0),v_1(0))$ is compactly supported in $\B_E(0,R)$ by Condition \ref{cond:ini}, we have that 
\begin{align*}
\E|v_1(t)|&\le R+\E|B_1(t)|\le R+c_d\sqrt{T}\\
\E|x_1(t)|&\le R+RT+\E\int_0^t|B_1(s)|ds\le R(1+T)+c_dT^{3/2}.
\end{align*}
For (b), by \eqref{eq-for-measure} for any $0\le s<t\le T$ we have that 
\begin{align*}
\E\big|\langle \mu_t^N,\, \phi\rangle - \langle \mu_s^N,\, \phi\rangle\big|
&\le \int_s^t\big|v\cdot\nabla_x\phi+\frac{1}{2}\Delta_v\phi, \mu_u^N\rangle\big|du\\
&+\E\Big|\int_s^t\frac{1}{N^2}\sum_{i\in\cN(u)}\sum_{j:j\neq i}\theta^\ep\big(x_i^N(u)-x_j^N(u)\big)\big[\phi\big(x_i^N(u),v_i^N(u)\big)+\phi\big(x_j^N(u),v_j^N(u)\big)\big] \, du\Big|\\
&+\big(\E|M_t^{\phi,D}-M_s^{\phi,D}|^2\big)^{1/2}+\big(\E|M_t^{\phi,J}-M_s^{\phi,J}|^2\big)^{1/2},
\end{align*}
where, as in the analysis of Section \ref{sec:emp} we can bound
\begin{align*}
\int_s^t\big|v\cdot\nabla_x\phi+\frac{1}{2}\Delta_v\phi, \mu_u^N\rangle\big|du&\le (K^\phi\|\phi\|_{C^1}+\|\phi\|_{C^2})|t-s|,\\
\E|M_t^{\phi,D}-M_s^{\phi,D}|^2&\le \frac{1}{N^2}\E\int_s^t\sum_{i\in\cN(u)}|\nabla_{v_i}\phi(x_i^N(u),v_i^N(u))|^2du \le \frac{\|\phi\|_{C^1}^2}{N}|t-s|,\\
\frac{1}{4}\E|M_t^{\phi,J}-M_s^{\phi,J}|^2&\le \frac{\|\phi\|_{L^\infty}^2}{N^3}\E\int_s^t\sum_{i\in\cN(u)}\sum_{j:j\neq i}\theta^\ep(x_i^N(u)-x_j^N(u))du\\
&\le \frac{\ep^{-d}\|\theta\|_{L^\infty}\|\phi\|_{L^\infty}^2}{N}|t-s|\le \|\theta\|_{L^\infty}\|\phi\|_{L^\infty}^2|t-s|,
\end{align*}
where $K^\phi$ denotes the maximal radius of the compact support of $\phi$. Further, upon dominating by the independent system, by \eqref{exp-decay} we have that
\begin{align*}
&\E\Big|\int_s^t\frac{1}{N^2}\sum_{i\in\cN(u)}\sum_{j:j\neq i}\theta^\ep\big(x_i^N(u)-x_j^N(u)\big)\phi\big(x_i^N(u),v_i^N(u)\big)\, du\Big|\\
&\le \E\int_s^t\theta^\ep(x_1(u)-x_2(u))|\phi(x_1(u),v_1(u))|du\\
&=\Gamma|t-s|\int_{\R^{4d}}\theta^\ep(x_1-x_2)|\phi(x_1,v_1)\big[Ce^{-|v_2|/C}1_{\{|v_2|\ge C\}}+\Gamma1_{\{|v_2|< C\}}\big]dx_1dx_2dv_1dv_2\\
&\le C(\Gamma,R,\phi)|t-s|,
\end{align*}
due to $\int \theta^\ep(x_1-x_2) dx_2=1$ for any $x_1$ and then $\phi(x_1,v_1)$ has compact support. Thus, we verified (b) with $\beta=1/2$.

\medskip
We now show that any weak subsequential limit of $\{\mu_t^N:t\in[0,T]\}$ is supported on the subset of $\cD_T(\cM_{+,1})$ that are absolutely continuous with respect to Lebesgue measure of $\R^{2d}$, for every $t\in[0,T]$, and its density is dominated from above by $\mathsf p(t,x,v)$ of \eqref{exp-decay}. To this end, we consider also the empirical measure of the independent system \eqref{free}:
\begin{align*}
\mu_t^{N,f}(dx,dv):=\frac{1}{N}\sum_{i=1}^N\delta_{\left(x_i(t), v_i(t)\right)}(dx, dv), \quad t\in[0,T].
\end{align*}
Since the continuous paths $\left(x_i(t),v_i(t), t\ge0\right)_{i=1}^\infty$ are independent and identically distributed, by Varadarajan's theorem (\cite[Theorem 11.4.1]{Du}), a.s. for every $t\in[0,T]$
\begin{align*}
\mu_t^{N,f}(dx,dv) \overset{N\to\infty}{\to} \ovl{\mu}_t^{f}(dx,dv):=\mathsf p(t,x,v)dxdv,
\end{align*}
where $\mathsf p(t,x,v)$ is the one-particle density \eqref{evol-den}, and $\ovl{\mu}_t^{f}$ is deterministic. Consider now the joint process of empirical measures
\begin{align*}
\left\{\mu_t^N,\mu_t^{N,f}: t\in[0,T]\right\}_{N\in\N}.
\end{align*}
Fix any weak subseqential limit $\ovl \mu_t$ of $\{\mu_t^N\}_N$ along some subsequences $\{N_k\}_k$. Since $\{\mu_t^{N,f}\}$ converges a.s. to $\ovl \mu_t^{f}$ along the full sequence, we have that the pair converges weakly along $\{N_k\}$:
\begin{align}\label{represent}
\big(\mu_t^{N_k},\mu_t^{N_k,f}: t\in[0,T]\big)\overset{k\to\infty}{\to} \big(\ovl \mu_t, \ovl\mu_t^f: t\in[0,T]\big).
\end{align}
By Skorohod's representation theorem, there exists an auxiliary probability space $(\wt\Omega, \wt \cF, \wt \P)$ and a sequence of random variables defined thereon with the same distributions as the ones in \eqref{represent} (by an abuse of notation we denote hereafter by the same letters) such that the convergence in \eqref{represent} is a.s. By the domination of paths between true and independent systems, a.s. for any open set $A\subset\R^{2d}$, $t\ge0$ and $k\in\N$,
\begin{align*}
0\le\mu_t^{N_k}(A)\le \mu_t^{N_k,f}(A).
\end{align*}
As $k\to\infty$, convergence in Skorohod topology implies that a.s. for any open set $A$, $t\in[0,T]$,
\begin{align*}
0\le\ovl{\mu}_t(A)\le \ovl{\mu}_t^{f}(A).
\end{align*}
Approximating a Borel set $B\subset\R^{2d}$ by a sequence of open sets $A_\ell\downarrow B$, we have that a.s. for any Borel sets $B$ and $t\in[0,T]$,
\begin{align*}
0\le\ovl{\mu}_t(B)\le \ovl{\mu}_t^{f}(B).
\end{align*}
Since for every $t$, $\ovl{\mu}_t^{f}$ is absolutely continuous with respect to Lebesgue measure, the same holds for $\ovl{\mu}_t$, a.s. Further, the density of $\ovl{\mu}_t^{f}$ is a.s. dominated by $\mathsf p(t,x,v)$, the density of $\ovl{\mu}_t^{f}$. Hence, we have proven the following proposition.

\begin{proposition}\label{ppn:subseq}
For any finite $T$ and $d\ge1$, the sequence $\{\mu_t^N(dx,dv):t\in[0,T]\}_{N\in\N}$ \eqref{empirical} of $\cD_T(\cM_{+,1})$-valued random variables is weakly relatively compact, and any subsequential limit law is concentrated on absolutely continuous paths, with nonnegative densities bounded above by $\mathsf p(t,x,v)$ of \eqref{exp-decay}.
\end{proposition}
Combining \eqref{eq-for-measure}, \eqref{mg-vanish-1}, \eqref{mg-vanish-2}, Proposition \ref{ppn:conv-nonlin} and Proposition \ref{ppn:subseq}, standard arguments can yield that any weak subsequential limit of $\{\mu_t^N(dx,dv):t\in[0,T]\}_{N\in\N}$ has a density $\ovl f(t,x,v)$ which is a (possibly random) weak solution of the \abbr{PDE} \eqref{pde}, in the sense of Definition \ref{def:weak}, provided we can remove the extra test function $\psi$ from the statement of Proposition \ref{ppn:conv-nonlin}, which is not present in the equation \eqref{eq-for-measure} satisfied by the empirical measure. We will also show in Section \ref{sec:unique} that weak solutions to \eqref{pde} are unique hence all subsequential limits must coincide and be nonrandom.

To this end, in place of $\psi(x,v)$ in \eqref{nonlinear}, consider a sequence $\{\psi^\Lambda(x,v)\}_{\Lambda\in\N}$ of the form $\psi^\Lambda(x,v)=\psi_1^\Lambda(x)\psi_2^\Lambda(v)$, where for $i=1,2$,
\begin{align*}
\psi_i^\Lambda: \R^d\to[0,1]
\end{align*}
is of class $C_c^\infty$ that is identically $1$ in $\B_E(0,\Lambda)$ and identically $0$ in $\B_E(0,2\Lambda)^c$. For any $i\neq j$, we have that 
\begin{align}\label{cut-remove}
&\int_0^T\theta^\ep\big(x_i^N(t)-x_j^N(t)\big)
\phi\big(t, x_i^N(t),v_i^N(t)\big) \, dt - 
\int_0^T\theta^\ep\big(x_i^N(t)-x_j^N(t)\big)
\phi\big(t, x_i^N(t),v_i^N(t)\big) \psi^\Lambda\big(x_j^N(t),v_j^N(t)\big)\, dt\nonumber\\
&=\int_0^T\theta^\ep\big(x_i^N(t)-x_j^N(t)\big)
\phi\big(t, x_i^N(t),v_i^N(t)\big) \big[1-\psi_1^\Lambda(x_j^N(t))\psi_2^\Lambda(v_j^N(t))\big]\, dt.
\end{align}
Since $|x_i^N(t)-x_j^N(t)|\le\ep$ in the support of $\theta^\ep$ and $x_i^N(t)$ is in the compact support of $\phi$, for $\Lambda$ large enough we have $\psi_1^\Lambda(x_j^N(t))=1$ and therefore
\begin{align*}
\eqref{cut-remove}=\int_0^T\theta^\ep\big(x_i^N(t)-x_j^N(t)\big)\phi\big(t, x_i^N(t),v_i^N(t)\big)%\psi_1^\Lambda(x_j^N(t))
\big[1-\psi_2^\Lambda(v_j^N(t))\big]\,dt.
\end{align*}
Furthermore, we have that
\begin{lemma}\label{lem:remove}
For any finite $T$ and $d\ge1$,
\begin{align*}
\lim_{\Lambda\to\infty}\sup_{N\in\N}\E\int_0^T\frac{1}{N^2}\sum_{i\in\cN(t)}\sum_{j:j\neq i}\theta^\ep\big(x_i^N(t)-x_j^N(t)\big)
\big|\phi\big(t, x_i^N(t),v_i^N(t)\big)\big[1-\psi_2^\Lambda(v_j^N(t))\big]\big|dt=0.
\end{align*}
\end{lemma}
\begin{proof}
We can upper bound by the independent system \eqref{free} and consider each term of the form 
\begin{align*}
\E\int_0^T\theta^\ep(x_1(t)-x_2(t))\big|\phi\big(t, x_1(t),v_1(t)\big)\big[1-\psi_2^\Lambda(v_2(t))\big]\big|dt.
\end{align*}
Since $1-\psi_2^\Lambda(v_2(t))$ is identically $0$ in $\B_E(0,\Lambda)$ and bounded above by $1$ elsewhere, utilizing \eqref{exp-decay} we have that 
\begin{align*}
&\E\int_0^T\theta^\ep(x_1(t)-x_2(t))\big|\phi\big(t, x_1(t),v_1(t)\big)\big[1-\psi_2^\Lambda(v_2(t))\big]\big|dt\\
&\le C(\Gamma, R, T)\int_{\R^{4d}}\theta^\ep(x_1-x_2)1_{\{|x_1|,|v_1|\le K^\phi\}}%1_{\{|x_2|\le K^\phi+1\}}
e^{-|v_2|/C}1_{\{|v_2|\ge \Lambda\}}\, dx_1dx_2dv_1dv_2\\
&\le C(\Gamma, R, T, \phi)\int_{\R^d}e^{-|v_2|/C}1_{\{|v_2|\ge \Lambda\}}\, dv_2\\
&\le C e^{-\Lambda/C},
\end{align*}
where $K^\phi$ denotes the maximal radius of the compact support of $\phi$, and $\int\theta^\ep(x_1-x_2)dx_2=1$ for every $x_1$. Note also $1/N^2$ cancels with the cardinality of the double sum, and the limit as $\Lambda\to\infty$ of the above expression is $0$, uniformly in $N$.
\end{proof}

\medskip
Since Proposition \ref{ppn:conv-nonlin} applies for $\psi=\psi^\Lambda$ for each $\Lambda>0$, \eqref{cut-remove} and Lemma \ref{lem:remove} yield a version of Proposition \ref{ppn:conv-nonlin} without the extra test function.
\begin{proposition}\label{ppn:conv-nonlin-new}
For any finite $T$ and $d\ge1$,
\begin{align*}
\lim_{\delta\to0}\limsup_{N\to\infty}&\,\E\Big|\int_0^T\frac{1}{N^2}\sum_{i\in\cN(t)}\sum_{j:j\neq i}\theta^\ep\big(x_i^N(t)-x_j^N(t)\big)
\phi\big(t, x_i^N(t),v_i^N(t)\big)\, dt\\
&-\int_0^T\int_{\R^{d}}dw\left\langle\eta^\delta(w-x_1)\eta^\delta(w-x_2)\phi(t, w,v_1), \;\mu_t^N(dx_1,dv_1)\mu_t^N(dx_2,dv_2)\right\rangle dt
\Big|=0.
\end{align*}
\end{proposition}

Equipped with \eqref{eq-for-measure} and Proposition \ref{ppn:subseq}, we complete the proof of subsequential convergence of the empirical measures. Since the sequence of laws $\{\mathsf P_N\}_N$ is tight, we consider any weakly converging subsequence $\mathsf P_{N_k}\Rightarrow \mathsf P^*$. Let us denote by $\mu^*$ a random variable taking values in $\cD_T([0,T];\cM_{+,1})$, having the law $\mathsf P^*$. By Proposition \ref{ppn:subseq}, we know that $\mu^*_t(dx,dv)=f^*(t,x,v)dxdv$ for some density $f^*(t,x,v)$ that is dominated from above by the function $\mathsf p(t,x,v)$ \eqref{exp-decay}. Given any $\iota>0$ and for any $\phi\in C_c^\infty([0,T]\times\R^{2d})$, we have that 
\begin{align}\label{eq:concen}
&\P\Big(\Big|-\langle \phi_T, \mu_T^*\rangle +\langle \phi_0, \mu_0^*\rangle + \int_0^T \left\langle \partial_t\phi+v\cdot \nabla_x\phi +\frac{1}{2} \Delta_v\phi,\, \mu_t^*\right\rangle dt \nonumber\\
&\quad \quad -\int_0^T\int_{\R^{d}}dw\left\langle\eta^\delta(w-x_1)\eta^\delta(w-x_2)\phi(t, w,v_1), \;\mu_t^*(dx_1,dv_1)\mu_t^*(dx_2,dv_2)\right\rangle dt\Big|>3\iota\Big) \nonumber\\
&\le \liminf_{N\to\infty}\P\Big(\Big|-\langle \phi_T, \mu_T^N\rangle +\langle \phi_0, \mu_0^N\rangle + \int_0^T \left\langle \partial_t\phi+v\cdot \nabla_x\phi +\frac{1}{2} \Delta_v\phi,\, \mu_t^N\right\rangle dt \nonumber\\
&\quad\quad -\int_0^T\int_{\R^{d}}dw\left\langle\eta^\delta(w-x_1)\eta^\delta(w-x_2)\phi(t, w,v_1), \;\mu_t^N(dx_1,dv_1)\mu_t^N(dx_2,dv_2)\right\rangle dt\Big|>3\iota\Big) \nonumber\\
&\le \liminf_N\P\left(|M_T^{N,\phi, D}|>\iota\right)+\liminf_N\P\left(|M_T^{N,\phi, J}|>\iota\right) \nonumber\\
&\quad\quad +\liminf_N\P\Big(\Big|\int_0^T\frac{1}{N^2}\sum_{i\in\cN(t)}\sum_{j:j\neq i}\theta^\ep\big(x_i^N(t)-x_j^N(t)\big)
\phi\big(t, x_i^N(t),v_i^N(t)\big)\, dt \nonumber\\
&\quad\quad -\int_0^T\int_{\R^{d}}dw\left\langle\eta^\delta(w-x_1)\eta^\delta(w-x_2)\phi(t, w,v_1), \;\mu_t^N(dx_1,dv_1)\mu_t^N(dx_2,dv_2)\right\rangle dt
\Big|
>\iota\Big),
\end{align}
where we used the identity \eqref{eq-for-measure}. The first two limits are zero by \eqref{mg-vanish-1}, \eqref{mg-vanish-2}. Upon taking another limit  $\delta\to0$ the third term also is zero, by Proposition \ref{ppn:conv-nonlin-new}. It remains to show that as $\delta\to0$, 
\begin{align}\label{eq:dom}
&\int_0^T\int_{\R^{d}}dw\left\langle\eta^\delta(w-x_1)\eta^\delta(w-x_2)\phi(t, w,v_1), \;\mu_t^*(dx_1,dv_1)\mu_t^*(dx_2,dv_2)\right\rangle dt\nonumber\\
&\to \int_0^T\int_{\R^{3d}} \phi(t,w,v_1)f^*(t,w,v_1)f^*(t,w,v_2)dwdv_1dv_2dt,
\end{align}
which we will do by dominated convergence. We note that for any $\delta>0$ and $w\in \R^d$,
\begin{align*}
&\left|\left\langle\eta^\delta(w-x_1)\eta^\delta(w-x_2)\phi(t, w,v_1), \;\mu_t^*(dx_1,dv_1)\mu_t^*(dx_2,dv_2)\right\rangle \right|\\
&\le \int_{\R^{4d}} \eta^\delta(w-x_1)\eta^\delta(w-x_2)|\phi(t, w,v_1)|f^*(t,x_1,v_1)f^*(t,x_2,v_2)dx_1dv_1dx_2dv_2\\
&\le C\|\phi\|_{L^\infty}\int_{\R^{4d}} \eta^\delta(w-x_1)\eta^\delta(w-x_2)1_{\left\{w\in\text{supp}(\phi)\right\}}e^{-(|v_1|+|v_2|)/C}dx_1dv_1dx_2dv_2\\
&\le C\|\phi\|_{L^\infty}1_{\left\{w\in\text{supp}(\phi)\right\}}
\end{align*}
where we used that $f^*(t,x,v)\le \mathsf p(t,x,v)$ which has exponential tail in $|v|$, uniformly in $x$. We also used that $\int_{\R^d} \eta^\delta(w-x_i) dx_i=1$, $i=1,2$, for any fixed $w$. The \abbr{RHS} of the preceding display is integrable in $dwdt$, 

On the other hand, as $\delta\to0$ the integrand of the \abbr{LHS} of \eqref{eq:dom}
\[
\left\langle\eta^\delta(w-x_1)\eta^\delta(w-x_2)\phi(t, w,v_1), \;\mu_t^*(dx_1,dv_1)\mu_t^*(dx_2,dv_2)\right\rangle \to \int \phi(t,w,v_1)f^*(t,w,v_1)f^*(t,w,v_2)dwdv_1dv_2
\]
for every fixed $t,w$, hence by the dominated convergence theorem we arrive at \eqref{eq:dom}.

Taking the limit as $\delta\to0$ on both sides of \eqref{eq:concen}, we conclude that $\P$-a.s., $f^*(t,x,v)$ satisfies 
\begin{align*}
\langle \phi_T, f^*(T,x,v)\rangle =&\langle \phi_0, f^*(0,x,v)\rangle + \int_0^T \left\langle \partial_t\phi+v\cdot \nabla_x\phi +\frac{1}{2} \Delta_v\phi,\, f^*(t,x,v)\right\rangle dt\\
&- \int_0^T\int_{\R^{3d}} \phi(t,w,v_1)f^*(t,w,v_1)f^*(t,w,v_2)dwdv_1dv_2dt,
\end{align*}
which is the weak formulation of \eqref{pde}. The $\P$-null set can be taken to be independent of $\phi$ by applying the above argument only to a countable dense subset of $C_c^\infty$.

\section{Uniqueness of weak solution}\label{sec:unique}
\begin{definition}\label{def:weak}
Assume $f_0\in L^1_+(\R^{2d})$ (the space of nonnegative integrable functions) satisfies Condition \ref{cond:ini}, we say that 
\begin{align}\label{class}
f\in L^\infty\big([0,T], L^1_+(\R^{2d})\big)
\end{align}
is a weak solution to \eqref{pde}
\begin{align*}
\begin{cases}
\partial_t f(t,x,v) &= - v\cdot\nabla_x f + \frac{1}{2}\Delta_v f - 2\int_{\R^d} f(t,x,v) f(t,x,v') dv'\\
f|_{t=0} &= f_0(x,v)
\end{cases}
, \quad (t,x,v)\in[0,T]\times\R^d\times\R^d
\end{align*}
if for any $\phi(t,x,v)\in C_c^\infty ([0,T]\times\R^{2d})$ and $t\in[0,T]$ we have that 
\begin{align*}
\int_{\R^{2d}}f(t,x,v)\phi(t, x,v)dxdv &= \int_{\R^{2d}}f_0(x,v)\phi(0, x,v)dxdv \\
&+\int_0^t\int_{\R^{2d}}\big( \partial_s+v\cdot\nabla_x  + \frac{1}{2}\Delta_v\big)\phi(s, x,v)f(s,x,v)dxdvds\\
&- 2\int_0^t\int_{\R^{3d}} f(s,x,v) f(s,x,v')\phi(s, x,v) dxdvdv'dt.
\end{align*}
\end{definition}

\begin{proposition}\label{ppn:unique}
For every finite $T$ and $d\ge1$, weak solitions to \eqref{pde} are unique.
\end{proposition}

\begin{proof}
As in Lemma \ref{lem:den}, we denote the operator
\begin{align*}
L^*=-v\cdot\nabla_x+\frac{1}{2}\Delta_v,
\end{align*}
and by $\{P(t)\}_{t>0}$ the Markov semigroup associated with $L^*$. It is a strongly continuous semigroup on $L^p(\R^{2d})$, for $1\le p<\infty$, see \cite[Corollary 2.6]{GT}, which we also denote by $P(t)=e^{tL^*}$. Let $f_1,f_2$ be two weak solutions of \eqref{pde} in the sense of Definition \ref{def:weak}, and set 
\[
\ovl f:= f_1-f_2. 
\]
Then, $\ovl f$ satisfies in weak form
\begin{align*}
\begin{cases}
(\partial_t-L^*)\ovl f(t,x,v) &= -2\int_{\R^d} f_1(t,x,v)\ovl f(t,x,v')dv'-2\int_{\R^d} \ovl f(t,x,v)f_2(t,x,v')dv'\\
\ovl f|_{t=0}&=0.
\end{cases}
\end{align*}
That is, for any $\phi(t,x,v)\in C_c^\infty([0,T]\times\R^{2d})$ and $t\in[0,T]$, we have that 
\begin{align*}
\int_{\R^{2d}}\ovl f(t,x,v)\phi(t, x,v)dxdv &= \int_0^t\int_{\R^{2d}}\big( \partial_s+L\big)\phi(s, x,v)\ovl f(s,x,v)dxdvds\\
&- 2\int_0^t\int_{\R^{3d}} \left[f_1(s,x,v) \ovl f(s,x,v')+\ovl f(s,x,v) f_2(s,x,v')\right] \phi(s, x,v)\, dv'dxdvds.
\end{align*}
A weak solution is also a mild solution. Indeed, for each $t$ taking $\phi(s,x,v)=\left(e^{(t-s)L}\wt \phi\right)(x,v)$ for $\wt \phi(x,v)\in C_c^\infty(\R^{2d})$, we obtain that 
\begin{align*}
&\int_{\R^{2d}}\ovl f(t,x,v)\wt\phi(x,v)dxdv = - 2\int_0^t\int_{\R^{3d}} \left[f_1(s,x,v) \ovl f(s,x,v')+\ovl f(s,x,v) f_2(s,x,v)\right] \left(e^{(t-s)L}\wt \phi\right)(x,v) dxdvdv'ds\\
&=-2\int_0^t\int_{\R^{2d}}e^{(t-s)L^*}\left(\int_{\R^d} f_1(s,\cdot,\cdot)\ovl f(s,\cdot,v')dv'+\int_{\R^d} \ovl f(s,\cdot,\cdot)f_2(s,\cdot,v')dv'\right)(x,v)\, \wt \phi(x,v)\, dxdvds.
\end{align*}
By the arbitrariness of $\wt \phi$, we see that $\ovl f$ satisfies
%By the equivalence between weak and mild formulations we have that 
\begin{align}\label{mild}
\ovl f(t,x,v)=-2\int_0^te^{(t-s)L^*}\left(\int_{\R^d} f_1(s,\cdot,\cdot)\ovl f(s,\cdot,v')dv'+\int_{\R^d} \ovl f(s,\cdot,\cdot)f_2(s,\cdot,v')dv'\right)(x,v)ds.
\end{align}
Since $P(t)=e^{tL^*}$ is a contraction in $L^p(\R^{2d})$, $1\le p\le \infty$, cf. \cite[pp. 255, last line]{Pr}, \cite[Lemma 2.4 (iv)]{GT}, in particular when $p=1$, we have by taking $L^1(\R^{2d})$-norm of \eqref{mild} that for any $t\in[0,T]$,
\begin{align}\label{intermediate}
\int_{\R^{2d}}\left|\ovl f(t,x,v)\right|dxdv&\le 2\int_0^t\left\|e^{(t-s)L^*}\left(\int_{\R^d} f_1(s,\cdot,\cdot)\ovl f(s,\cdot,v')dv'+\int_{\R^d} \ovl f(s,\cdot,\cdot)f_2(s,\cdot,v')dv'\right)\right\|_{L^1(\R^{2d})}ds  \nonumber\\
&\le  2\int_0^t\left\|\int_{\R^d} f_1(s,\cdot,\cdot)\ovl f(s,\cdot,v')dv'+\int_{\R^d} \ovl f(s,\cdot,\cdot)f_2(s,\cdot,v')dv'\right\|_{L^1(\R^{2d})}ds  \nonumber\\
&\le 2\int_0^t \int_{\R^{3d}} f_1(s,x,v)\left|\ovl f(s,x,v')\right|dxdvdv'ds+2\int_0^t\int_{\R^{3d}} \left|\ovl f(s,x,v)\right|f_2(s,x,v')dxdvdv'ds.
\end{align}
By a simple comparison, we have that 
\begin{align*}
0\le f_i(t,x,v)\le g(t,x,v), \quad i=1,2,
\end{align*}
where $g$ is the unique (weak) solution to
\begin{align*}
(\partial_t-L^*) g(t,x,v) &=0.\\
g|_{t=0}&=f_0(x,v).
\end{align*}
Further, $g(t)=P(t)f_0$ is precisely the function $\mathsf p(t,x,v)$ in \eqref{exp-decay} hence, since $f_0$ satisfies Condition \ref{cond:ini}, we have that for some finite constant $C(T,\Gamma,R,d)$ and uniformly for all $t\in[0,T]$ and $x\in\R^d$
\begin{align}\label{v-marginal}
\int_{\R^d} f_i(t,x,v)dv\le \int_{\R^d} g(t,x,v)dv\le  C(T,\Gamma,R,d), \quad i=1,2.
\end{align}
Hence, we obtain by \eqref{intermediate}-\eqref{v-marginal} that for any $t\in[0,T]$,
\begin{align*}
\left\|\ovl f(t)\right\|_{L^1(\R^{2d})}&\le C\int_0^t \int_{\R^{2d}}\left|\ovl f(s,x,v')\right|dxdv'ds+ C\int_0^t\int_{\R^{2d}} \left|\ovl f(s,x,v)\right|dxdvds\\
&=C\int_0^t\left\|\ovl f(s)\right\|_{L^1(\R^{2d})}ds.
\end{align*}
By Gronwall's lemma, we have that $\ovl f=f_1-f_2=0$ in the class \eqref{class}.
\end{proof}

\medskip
Since any weak subsequential limit obtained in Proposition \ref{ppn:subseq} is supported in the class \eqref{class}, by Proposition \ref{ppn:unique} we have uniqueness of subsequential limits hence convergence in law of the full sequence $\{\mu^N_t\}_{N\in\N}$ to the unique weak solution of the \abbr{PDE}, as stated in Theorem \ref{thm:main}. Since the limit is deterministic, we actually have convergence in probability.

\medskip
\noindent
{\textbf{Acknowlegments.}} The author thanks Prof. F. Flandoli for valuable discussions, corrections and pointers. The author thanks two anonymous referees whose constructive feedback significantly improved the presentation of the paper. This research was supported in part by a CRM-Pisa Junior Visiting Position 2020-22.

\end{document}